\documentclass[10pt]{article}

\usepackage[table,xcdraw]{xcolor}
\usepackage{hyperref}
\usepackage{natbib}

\usepackage[normalem]{ulem}
\usepackage{enumitem}
\usepackage{subcaption}

\usepackage{amsmath, amssymb, amsfonts, mathrsfs}
\usepackage{mathtools}
\usepackage{bbm, bm}
\usepackage{dsfont}
\usepackage[amsmath, thmmarks]{ntheorem}

\DeclarePairedDelimiter\rb{(}{)}
\DeclarePairedDelimiter\curbr{\lbrace}{\rbrace}
\DeclarePairedDelimiter\cb{[}{]}

\DeclarePairedDelimiter\cbr{(}{]}
\DeclarePairedDelimiter\abs{\lvert}{\rvert}
\DeclarePairedDelimiter\floorfct{\lfloor}{\rfloor}
\DeclarePairedDelimiter\ceilfct{\lceil}{\rceil}

\newcommand{\bigO}[1]{ \mathcal{O}\rb*{#1} }

\newcommand{\dd}{\, \mathrm{d}}

 \renewcommand{\epsilon}{\varepsilon}
         \renewcommand{\phi}{\varphi}

\let\E\relax \newcommand{\E}[2][]{\mathbb{E}_{#1} \left( #2 \right)}

\newcommand{\Unif}{\mathcal{U}}

\newcommand{\ind}[1]{ \mathbf{1}_{ #1 } }

\newcommand{\0}{\bm{0}}

\newcommand{\Borel}{\mathcal{B}}
\newcommand{\I}{\mathbb{I}}
\newcommand{\N}{\mathbb{N}}

\newcommand{\R}{\mathbb{R}}

\newcommand{\Copula}{\mathcal{C}}

\newcommand{\CB}{\mathcal{CB}}
\newcommand{\B}{\mathcal{B}}

\newcommand{\showchanges}{1} 
\if\showchanges1
    \newcommand{\old}[1]{\textcolor{red}{\sout{#1}}}
    
\else
    \newcommand{\old}[1]{}
    
\fi

\numberwithin{equation}{section}

\newtheorem{theorem}{Theorem}[section]

\newtheorem{lemma}[theorem]{Lemma}
\newtheorem{proposition}[theorem]{Proposition}
\newtheorem{assumption}[theorem]{Assumption}

\newtheorem{definition}[theorem]{Definition}
\newtheorem{example}[theorem]{Example}

\newtheorem{remark}[theorem]{Remark}

\makeatletter
\newtheoremstyle{MyNonumberplain}%
  {\item[\theorem@headerfont\hskip\labelsep ##1\theorem@separator]}%
  {\item[\theorem@headerfont\hskip\labelsep ##3\theorem@separator]}
\makeatother

\theoremstyle{MyNonumberplain}
\theorembodyfont{\normalfont}
\theoremsymbol{\ensuremath{\square}}
\newtheorem{proof}{Proof}

\providecommand{\keywords}[1]
{
  \small	
  \textit{Keywords}: \hangindent=7em \hangafter=1 #1
}

\providecommand{\msc}[1]
{
  \small	
  \textit{2020 MSC}: #1
}

\title{Estimating Conditional Distributions via Sklar's Theorem and Empirical 
Checkerboard Approximations, with Consequences to \\
Nonparametric Regression}

\date{}

\author{Kai Schärer\thanks{Department for Artificial Intelligence \& Human Interfaces (AIHI), University of Salzburg, Salzburg, Austria. \url{kai.schaerer@plus.ac.at}}, 
Wolfgang Trutschnig\thanks{Department for Artificial Intelligence \& Human Interfaces (AIHI), University of Salzburg, Salzburg, Austria. \url{wolfgang@trutschnig.net}, Corresponding Author}}

\begin{document}

\maketitle

\begin{abstract}
We tackle the natural question of whether it is possible to estimate 
conditional distributions via Sklar's theorem by 
separately estimating the conditional distributions of the underlying co\-pula 
and the marginals. Wor\-king with so-called empirical 
checkerboard/Bernstein approximations with suitably chosen  
resolution/degree, we first show that uniform weak convergence to the 
true underlying copula can be established under very mild regularity 
assumptions. Building upon these results and 
plugging in the univariate empirical marginal distribution functions we then 
provide an affirmative answer to the afore-mentioned question and prove 
strong consistency of the resul\-ting estimators for the conditional 
distributions. Moreover, we show that aggregating our estimators allows to 
construct consistent nonparametric estimators for the mean, the quantile, and the 
expectile regression function, and beyond. 
Some simulations illustrating the performance of the estimators and a 
real data example complement the established theoretical results.
\end{abstract}


\msc{60E05, 62G05, 62G08, 62H05}

\section{Introduction}

Allowing to flexibly model/capture the relationship between a response variable $Y$ and a vector $\mathbf{X}$ of covariates without imposing restrictive parametric assumptions, nonparametric regression techniques play a central role in modern statistics and all quantitatively oriented research fields. 
Classical regression methods typically focus on a single point estimate, such as the conditional mean, and cover kernel-based estimators going back to \citet{nadaraya64} and \citet{watson64} as well as local polynomial modelling techniques due to \citet{fan96}. Although mean regression may be considered one of the most used tools in data analysis, it only provides limited information on the conditional distribution and may overlook important aspects, such as skewness, heavy tails, and heteroscedasticity. 

One natural extension, aiming to address these limitations is nonparametric quantile regression. This method directly estimates conditional quantile functions and thus may provide a more informative description of the relationship between the response variable and the covariates. Key techniques of nonparametric quantile regressions are the linear quantile estimator due to \citet{koenker78}, estimators due to \citet{chaudhuri91}, and a method using density estimation due to \citet{huang18}. Previous work of copula based nonparametric quantile regression focuses on vine copulas (as special case
of pair-copula constructions) and is, e.g., discussed in \citet{tepegjozova2021}.

Another alternative to mean regression is expectile regression, which can be 
considered as asymmetric version of mean regression. It is a powerful tool for capturing tail risk and asymmetric conditional behaviour, with the asymmetric least square estimator studied by \citet{newey87} being the first widely recognised expectile regression method. 
As in the mean and the quantile setting, nonparametric methods were developed, such as the kernel-based estimator studied by \citet{yao96} and local polynomial modelling techniques analyzed by \citet{adam22}. Contrary to (nonparametric) mean 
and quantile regression, there is little literature on nonparametric expectile 
regression, not to mention unified approaches being able to simultaneously accommodate means, quantiles, and expectiles.

Obviously, nonparametrically estimating the conditional distribution functions 
and subsequently aggregating the latter, reconciles all three of the 
afore-mentioned research questions. The naive kernel-based distributional regression 
considered by \citet{fan93} faces various issues, such as non-monotonicity \citet{hall99}, which, in turn, was recently tackled by \citet{das20}. 

Focussing on the continuous setting and considering the fact that, according to Sklar's
theorem (see \citet{sklar59}), copulas constitute the link between joint distribution functions and their univariate marginals, it seems natural to tackle (mean, quantile, expectile) regression by modelling/estimating the underlying copula and the marginal distributions separately. Following this idea, \citet{noh13} introduced a novel 
semiparametric estimator for the regression function in the context of 
a high-dimensional predictor. As illustrated and discussed by \citet{dette14}, however, even in the bivariate setting, working with parametric families, any misspecification of the copula class `may lead to invalid results'.

In the current paper we revisit the idea of tackling regression and, more generally, 
the estimation of conditional distributions, via Sklar's theorem. 
We focus on the bivariate setting and propose a novel distributional regression 
method based on nonparametric estimation of the underlying co\-pu\-la in terms of so-called 
checkerboard and Bernstein approximations of the empirical 
copula (to which we will 
refer to as empirical checkerboard and empirical Bernstein approximations, respectively). 
Using recently established results on weak conditional convergence of (empirical) checkerboard approximations (see \citet{ansari26}), we prove that (under quite mild regularity conditions) our approach leads to strongly consistent and computationally fast 
estimators for conditional distributions. Moreover, we show that aggregating 
these estimators yields consistent estimators for mean, quantile and expectile 
regression and beyond. 
Complementing the checkerboard approach, we then focus on empirical Bernstein approximations
and derive analogous consistency results. 
Since, to the best of our knowledge, the rate of convergence both, for 
empirical checkerboard and Bernstein approximations is still unknown, we 
test the performance 
of our estimators via a small simulation study and compare their performance 
with other well-known nonparametric kernel-based regression methods. \\

The rest of this article is organized as follows. Section~\ref{sec:notation} gathers notation and preliminaries required for all subsequent sections. Section~\ref{sec:uniform_convergence} and Section~\ref{sec:uniform_convergence_bernstein} establish convergence results for the checkerboard and the Bernstein approach, respectively. As shown in Section~\ref{sec:regression}, these results then open the door 
for constructing consistent estimators for all three types of regression functions mentioned before. 
A simulation study illustrating the performance of the established estimators in Section~\ref{sec:simulation}, and a real-world example 
with insurance data in Section~\ref{sec:real_world} round off the paper.



\section{Notation and Preliminaries}\label{sec:notation}

For every metric space $(S,d)$ the Borel $\sigma$-field will be denoted by
$\Borel(S)$, the open ball of radius $r>0$ around $x \in S$ by $B_r(x)$. 
The closed unit interval will be denoted by $\I := [0,1]$, the Lebesgue measure on $\Borel(\I^2)$ and $\Borel(\I)$ by $\lambda_2$ and $\lambda$, respectively. 
For every $N \geq 2$ we will work with the partition $$I_1^N := \cb*{0, \tfrac{1}{N}}, I_2^N := \cbr*{\tfrac{1}{N}, \tfrac{2}{N}}, \dotsc, I_N^N := \cbr*{\tfrac{N-1}{N}, 1}$$ of $\I$ and (in order to simplify some formulas in the rest of the paper) 
set $I^N_0 := \emptyset =: I^N_{N+1}$. We define    
$Q^N_{i,j} := I^N_i \times I^N_j \subset \I^2$ for all 
$i, j \in \curbr{0, \dotsc, N+1}$ and, given $(x,y) \in \I^2$, 
let $i_0^N(x), j_0^N(y) \in \curbr{1, \dotsc, N}$ denote the unique indices with 
$(x, y) \in Q^N_{i_0^N(x), j_0^N(y)}$.


$\Copula$ will denote the family of all 
bivariate copulas, $\mathcal{P}$ the family of all doubly stochastic measures 
on $\mathcal{B}(\I^2)$ and $\Pi$ the product copula. 
It is well-known that every copula $A$ corres\-ponds to 
a unique doubly stochastic measure $\mu_A$ and vice versa, where the correspondence is given through the identity $A(x,y)=\mu_A([0,x] \times [0,y])$ for all $x,y \in \I$. 
We will refer to a copula $A$ as absolutely continuous, if the corresponding 
doubly stochastic measure $\mu_A$ is absolutely continuous. 
As usual, $d_\infty$ will denote the uniform metric on $\mathcal{C}$, i.e., 
$d_\infty(A,B) : = \max_{(x,y) \in [0,1]^2} |A(x,y) - B(x,y) |$. 
It is well known that $(\mathcal{C}, d_\infty)$ is a compact metric space. 
For random variables $X,Y$ we will write $(X,Y) \sim H$ if 
$(X,Y)$ has distribution function $H$ and we will write $(U,V) \sim A \in \Copula$ if 
$A$ is the distribution function of $(U,V)$ restricted to $\I^2$; in the latter
case we obviously have that $U$ and $V$ are uniformly distributed on $\I$. 
For more background on copulas we refer to \citet{sempi15} as well as to 
\citet{nelsen06}.

Suppose that $(x_1,y_1),\ldots, (x_n,y_n)$ is a sample of $(X,Y) \sim H$ with 
$H$ being continuous and suppose that $A$ is the underlying copula. Furthermore, let 
$H_n$ denote the bivariate empirical distribution function and $F_n,G_n$ the univariate empirical marginal distribution functions. 
Sklar's Theorem implies the existence of a unique subcopula (see \cite{nelsen06}) $E'_n: \operatorname{Range}(F_n) \times \operatorname{Range}(G_n) \to \operatorname{Range}(H_n)$ fulfilling 
\begin{align*}
H_n(x,y)=E'_n(F_n(x),G_n(y)),
\end{align*}
for all $(x,y) \in \mathbb{R}^2$. 
Since $H$ is continuous, ties can only occur with probability $0$, so 
$\operatorname{Range}(F_n) = \operatorname{Range}(G_n)=\{0,\frac{1}{n}, \ldots, \frac{n-1}{n},1\}$ holds almost surely. 
There are infinitely many ways to extend $E_n'$ to a copula $E_n$, in the sequel we 
will only work with the bilinear interpolation (see \cite{junker21,nelsen06}) and refer to this extension $E_n$ simply as empirical copula of the sample $(x_1,y_1),\ldots, (x_n,y_n)$.

A mapping $K:\mathbb{R} \times \mathcal{B}(\mathbb{R}) \to [0,1]$ is called a Markov kernel from $\mathbb{R}$ to $\mathcal{B}(\mathbb{R})$ if $x \mapsto K(x,B)$ is measurable for every fixed $B \in \mathcal{B}(\mathbb{R})$ and $B \mapsto K(x, B)$ is a probability measure for every fixed $x \in \mathbb{R}$. A Markov kernel $K:\mathbb{R} \times \mathcal{B}(\mathbb{R}) \to [0,1]$ is called regular conditional distribution of a 
random variable $Y$ given (another random variable) $X$ if for every $B \in \mathcal{B}(\mathbb{R})$ 
\begin{align*}
	K(X(\omega),B) = \mathbb{E}\rb*{ \ind{B} \circ Y \mid X}(\omega)
\end{align*}
holds for $\mathbb{P}$-almost every $\omega \in \Omega$, whereby $\ind{B}$ denotes the indicator function of the set $B$. It is well known that a regular conditional distribution 
of $Y$ given $X$ exists and that is unique $\mathbb{P}^X$-almost everywhere (with $\mathbb{P}^X$ 
denoting the 
distribution of $X$, i.e., the push-foward of $\mathbb{P}$ under $X$). 
For every bivariate distribution function $H$ the corresponding regu\-lar conditional distribution (i.e., the regular conditional distribution of $Y$ given $X$ in the case that $(X,Y)\sim H$) will be denoted by $K_H(\cdot, \cdot)$ and, accordingly, 
$K_A(\cdot,\cdot)$ will denote the regular conditional distribution of the copula $A \in \Copula$. In the sequel we will refer to $K_H(\cdot, \cdot)$ and $K_A(\cdot, \cdot)$
simply as (versions of the) Markov kernel of $H$ and $A$, respectively.

Suppose that $(X,Y)$ has distribution function $H$.
Given $G \in \mathcal{B}(\mathbb{R}^2)$ and letting $G_x:=\{y \in \I: \,(x,y) \in G\}$ denote the 
$x$-cut of $G$ for every $x \in \mathbb{R}$, the disintegration theorem (see \cite{kallenberg21})
implies that  
\begin{align}\label{eq:disint}
\int_\mathbb{R} K_H(x,G_x) \dd{\mathbb{P}^X(x)} = \mathbb{P}^{(X,Y)}(G)
\end{align}
holds. As particular case, for $A \in \Copula$ and arbitrary 
rectangles $G=E_1 \times E_2$ with $E_1,E_2 \in \mathcal{B}(\I)$ we have  
\begin{align*}
\int_{E_1} K_A(x,E_2) \dd{\lambda(x)} = \mu_A(E_1 \times E_2).
\end{align*}

Again suppose that $(X,Y) \sim H$ with $H$ being continuous, let $F,G$ denote the corresponding
marginal distribution functions and $A$ the (unique) underlying copula. 
Then according to \citet{mroz23}, considering the Markov kernel of $A$ and 
plugging in the marginal distribution functions by setting 
\begin{align}\label{eq:sklar.for.kernels}
K_H(x, (-\infty, y]) := K_A\left(F(x), [0, G(y)]\right)   
\end{align}
yields a Markov kernel $K_H(\cdot,\cdot)$ of $H$. We will use this interrelation when 
extending the convergence of the so-called empiricial checkerboard/Bernstein approximations 
to the convergence of the conditional distributions of $Y$ given $X$.

While Markov kernels of copulas are only unique $\lambda$-almost everywhere, 
it is straightforward to verify that every copula $A$ allows at most one version of the Markov kernel 
such that the function $(x, y) \mapsto K_A(x, [0, y])$ is continuous on $\I^2$. 
In case such a kernel exists, we will refer to it as the (unique) 
continuous Markov kernel of $A$ and say that $A$ allows a continuous Markov kernel. 
If $\mu_A$ is absolutely continuous with density $k_A$, then setting 
\begin{equation}\label{eq:int.density}
  K_A(x,[0,y]) := \int_{[0,y]} k_A(x,s) d\lambda(s), \quad x \in \I   
\end{equation}
obviously defines a Markov kernel of $A$. For more details and properties of regular conditional distributions and disintegration see \citet{kallenberg21} and \citet{klenke20}.

Given copulas $A,A_1,A_2\ldots$ we will say that the sequence $(A_n)_{n \in \mathbb{N}}$
converges weakly conditional to $A$, if for $\lambda$-almost every 
$x \in \I$ we have that the sequence $(K_A(x,\cdot))_{n \in \mathbb{N}}$ 
of probability measures on $\mathcal{B}(\I)$ converges weakly to the probability
measure $K_A(x,\cdot)$. Although in several classes like the families of 
bivariate Archimdean or extreme-value copulas weak conditional convergence 
is equivalent to convergence w.r.t. $d_\infty$ (see \citet{kasper21}), in general 
these two notions are not equivalent and weak conditional convergence implies
convergence w.r.t. $d_\infty$, but not vice versa. 
The main results of this paper (Theorem \ref{eq:uniform_convergence} and 
Theorem \ref{thm:uniform_convergence_bernstein}) build upon the following 
even stronger notion of convergence of copulas, which we will refer to as uniform conditional convergence: 
\begin{definition}
A sequence $(A_n)_{n \in \mathbb{N}}$ of copulas
converges uniformly conditional to a copula $A$, if there are 
versions $K_A,K_{A_1},K_{A_2},\ldots$ of the Markov kernels such that
\begin{align*}
 \lim_{n \rightarrow \infty} \sup_{(x,y) \in \I^2} \abs*{ K_{A_n}(x, [0, y]) - K_A(x, [0, y]) } =0.    
\end{align*}
\end{definition}

In what follows we will work with so-called checkerboard and Bernstein approximations.
Before focusing on the main results in the next sections, we first recall their construction and properties.


\subsection{Checkerboard Approximations}

Suppose that $2 \leq N \in \mathbb{N}$. Then the $N$-checkerboard approximation $\CB_N(A) \in \Copula$ of $A \in \Copula$ is by definition 
the absolutely continuous copula with (a version of the) density $k_{\CB_N(A)}$ 
given by (see \citet{li98}) 
\begin{align}\label{eq:density_cb}
    k_{\CB_N(A)}(x, y)
    &= N^2 \sum_{i,j=1}^{N} \mu_A\rb*{Q^N_{i,j}} \ind{Q^N_{i,j}}(x, y) \nonumber \\
    &\begin{multlined}[t][9cm]
        = N^2 \sum_{i,j=0}^N A\rb*{\tfrac{i}{N}, \tfrac{j}{N}} \rb*{\ind{I^N_i}(x) - \ind{I^N_{i+1}}(x)} \rb*{\ind{I^N_j}(y) - \ind{I^N_{j+1}}(y)}
    \end{multlined}
\end{align}
for all $(x, y) \in \I^2$. In the sequel, following \citet{griessenberger22} and \citet{junker21}
we will refer to $N$ as the resolution of the checkerboard approximation.
Integrating eq.~\eqref{eq:density_cb} according to eq.~\eqref{eq:int.density} directly yields the following 
expression for the Markov kernel $K_{\CB_N(A)}(\cdot,\cdot)$ of $\CB_N(A)$, which will prove useful in the sequel:
\begin{equation}\label{eq:checkerboard_approx_conditional_df}
\begin{aligned}
    K_{\CB_N(A)}(x, \cb{0,y})
    &= N^2 \sum_{i,j=1}^{N} \mu_A\rb*{Q^N_{i,j}} \ind{I^N_i}(x) \int_{\cb{0,y}} \ind{I^N_j} \dd{\lambda} \\
    &\begin{multlined}[t][9cm]
        = N^2 \sum_{i,j=0}^N A\rb*{\tfrac{i}{N}, \tfrac{j}{N}} \rb*{\ind{I^N_i}(x) - \ind{I^N_{i+1}}(x)} \int_{\cb{0,y}} \rb*{\ind{I^N_j} - \ind{I^N_{j+1}}} \dd{\lambda}.
    \end{multlined}
\end{aligned}
\end{equation}

\noindent Using the indices $i_0^N(x),j^N_0(x)$ as mentioned 
at the begging of this section, the last expression in eq.~\eqref{eq:checkerboard_approx_conditional_df} further boils down to 
\begin{align}\label{eq:checkerboard_approx_conditional_df_2}
    K_{\CB_N(A)}(x, [0, y]) 
    &= 
    N(j_0^N(y) - yN) \cb*{ A\rb*{ \tfrac{i_0^N(x)}{N}, \tfrac{j_0^N(y) - 1}{N}} - A\rb*{ \tfrac{i_0^N(x) - 1}{N}, \tfrac{j_0^N(y) - 1}{N}}} \nonumber \\
    &+ N(yN + 1 - j_0^N(y)) \cb*{ A\rb*{ \tfrac{i_0^N(x)}{N}, \tfrac{j_0^N(y)}{N}} - A\rb*{ \tfrac{i_0^N(x) - 1}{N}, \tfrac{j_0^N(y)}{N}}},
   \end{align}
which implies that $\frac{1}N K_{\CB_N(A)}(x, [0, y])$ can be interpreted as convex 
combination of the differences 
$$ A\rb*{ \tfrac{i_0^N(x)}{N}, \tfrac{j_0^N(y) - 1}{N}} - A\rb*{ \tfrac{i_0^N(x) - 1}{N}, \tfrac{j_0^N(y) - 1}{N}}, \quad   
A\rb*{ \tfrac{i_0^N(x)}{N}, \tfrac{j_0^N(y)}{N}} - A\rb*{ \tfrac{i_0^N(x) - 1}{N}, \tfrac{j_0^N(y)}{N}}.
$$
Notice that using eq.~\eqref{eq:checkerboard_approx_conditional_df_2} and the triangle inequality directly yields the following inequality for arbitrary $A_1,A_2 \in \Copula$:  
\begin{align}\label{ineq:check}
    \abs*{ K_{\CB_N(A_1)}(x, [0, y]) -  K_{\CB_N(A_2)}(x, [0, y]) } \leq  
      2Nd_\infty(A_1,A_2).
\end{align}
The importance of checkerboard approximations rests in particular on 
the fact (established via Markov operators in \cite{li98}) that for every 
copula $A \in \Copula$ the sequence $(\CB_N(A))_{N \in \mathbb{N}}$ converges weakly 
conditional to $A$ for $N \rightarrow  \infty$ (also see Section 4 in \cite{ansari26}). 
In Section 3 we will refine this result and show that if $A$ allows a continuous 
Markov kernel then the sequence $(\CB_N(A))_{N \in \mathbb{N}}$ 
even converges uniformly conditional to $A$.


\subsection{Bernstein Approximations}

Bernstein copulas (see, e.g., \cite{sanchetta04, segers17} 
and the references therein) are another popular approximation method for copulas
based on partitions of unity studied in a more general context in \citet{li98}. 
While for checkerboard approximations the partition of unity is constituted by 
$\ind{I^N_1}, \ldots,\ind{I^N_N}$, Bernstein approximations build upon  
the Bernstein polynomials. 
Recall that the $s$-th Bernstein polynomial of degree $N \in \mathbb{N}$ is given by 
\[ p_{N, s}(u) = \begin{cases}
    \binom{N}{s} u^s (1- u)^{N-s}, & s \in \curbr{0, \dotsc, N}, \\
    0, & s \notin \curbr{0, \dotsc, N}
\end{cases},  \quad u \in \I, \]
where the latter case is just considered for simplifying calculations in the sequel. 
For $s \in \curbr{1, \dotsc, N}$ it is straightforward to verify that the derivative is given by
\[ p_{N, s}'(u) = N \rb*{ p_{N-1, s-1}(u) - p_{N-1, s}(u) }\]
for all $u \in \I$. 

The following auxiliary lemma states that we can bound the sum of the Bernstein polynomials, when dropping indices around $xN$. 
\begin{lemma}\label{lem:bernstein_chernoff}
Suppose that $u \in \I$ and $N \in \N$, consider $t\geq 0$ and set  
\[ V := \curbr{ 0, \dotsc, \floorfct*{ uN - t }, \ceilfct*{uN + t}, \dotsc, N }, \]
where $\floorfct{\cdot}$ and $\ceilfct{\cdot}$ denote the floor and ceiling functions, respectively. Then the inequality 
\[ \sum_{i \in V} p_{N, i}(u) \leq 2\exp \rb*{ - \tfrac{2t^2}{N} } \]
holds.
\end{lemma}
\begin{proof}
Letting $X$ denote a random variable following a binomial distribution $\operatorname{Bin}(N, u)$, we obviously have 
\[ \sum_{i \in V} p_{N, i}(u) = \mathbb{P}({ \abs*{ X - \E{X} } \geq t }). \]
The inequality now follows directly from the Chernoff-Hoeffding bound \citep{hoeffding63}.
\end{proof}

For every copula $A \in \Copula$, its $N$-Bernstein approximation $\B_N(A)$ 
is given by
$$
\B_N(A)(x, y) := \sum_{i, j = 1}^N A \rb*{ \tfrac{i}{N}, \tfrac{j}{N} } p_{N, i}(x) p_{N, j}(y)
$$
for all $x,y \in \I$. 
Using smoothness and considering the first derivative w.r.t. $x$, the continuous version of the Markov kernel $K_{B_N(A)}$ of $\B_N(A)$ is given by (equality of the first and the second line is straightforward to verify)
\begin{equation}\label{eq:markov_kernel_bernstein}
\begin{aligned}
    K_{\B_N(A)}(x, [0, y]) :&= N \sum_{i, j = 1}^N A \rb*{ \tfrac{i}{N}, \tfrac{j}{N} } (p_{N-1, i-1}(x) - p_{N-1, i}(x)) p_{N, j}(y) \\
    &= N \sum_{i, j = 1}^N \rb*{ A \rb*{ \tfrac{i}{N}, \tfrac{j}{N} } - A \rb*{ \tfrac{i-1}{N}, \tfrac{j}{N} } } p_{N-1, i-1}(x) p_{N, j}(y)
\end{aligned}
\end{equation}
for all $(x, y) \in \I^2$.
Using the first or the second line in eq.~\eqref{eq:markov_kernel_bernstein} and the triangle inequality directly yields the following inequality (which is the Bernstein version of ineq.~\eqref{ineq:check} 
for arbitrary $A_1,A_2 \in \Copula$:  
\begin{align}\label{ineq:bern}
    \abs*{ K_{\B_N(A_1)}(x, [0, y]) -  K_{\B_N(A_2)}(x, [0, y]) } \leq  
      2Nd_\infty(A_1,A_2).
\end{align}
Translating the (Markov operator based) 
results from \cite{li98} to Markov kernels, analogous 
to checkerboard approximations we have that for every 
copula $A$ the sequence $(\B_N(A))_{N \in \mathbb{N}}$ converges weakly 
conditional to $A$ for $N  \rightarrow  \infty$.
As for checkerboard approximation in Section 3, in Section 4 we will refine this 
result and show that if $A$ allows a continuous 
Markov kernel then the sequence $(\B_N(A))_{N \in \mathbb{N}}$ 
even converges uniformly conditional to $A$.


\section{Uniform Conditional Convergence of the Empirical Checkerboard Approximations}\label{sec:uniform_convergence}

In this section we show that the checkerboard approximation $\CB_{N(n)}(E_n)$ of the empirical copula $E_n$ with appropriately chosen resolution $N=N(n)$ 
together with the empirical marginal distributions, establishes a strongly 
consistent estimator for the conditional distributions. 
In order to keep the wording simple, in the sequel we will 
refer to $\CB_{N(n)}(E_n)$ as the \emph{empirical checkerboard approximation} with
resolution $N(n)$.  We first state this sections's 
main theorem on the uniform conditional convergence of the empirical checkerboard copula and then prove the result in several steps. 
Recall that we say that $A$ has a continuous Markov kernel if, and only if there exists 
a (version of) the Markov kernel such that the function 
$(x, y) \mapsto K_A(x, [0, y])$ is continuous on $\I^2$.

\begin{theorem}\label{thm:uniform_convergence}
Suppose that $(X_1, Y_1), (X_2, Y_2), \dotsc$ is a random sample from $(X,Y)$ and assume 
that $(X,Y)$ has continuous distribution function $H$, mar\-gi\-nals $F$ and $G$, and underlying (unique) copula $A$. Furthermore suppose that $A$ allows a continuous Markov kernel $K_A$ and define $K_H$ according to eq.~\eqref{eq:sklar.for.kernels}. 
Setting $N(n) := \floorfct{n^s}$ for some $s \in \rb*{0, \frac{1}{2}}$, we have that
\begin{equation}\label{eq:uniform_convergence}
    \lim_{n \rightarrow \infty} \sup_{(x, y) \in \R^2} \abs*{K_{\CB_{N(n)}(E_n)}(F_n(x), [0, G_n(y)]) - K_H(x, (-\infty, y])} = 0
\end{equation}
almost surely. 
\end{theorem}

In order to prove Theorem~\ref{thm:uniform_convergence} 
we first derive uniform conditional convergence of the checkerboard approximation, 
then plug-in the empirical co\-pu\-la to obtain uniform conditional 
convergence of the empirical checkerboard $(\CB_{N(n)}(A))_{n \in \mathbb{N}}$, 
and finally extend the convergence via eq.~\eqref{eq:sklar.for.kernels} to $K_H(\cdot,\cdot)$.

\begin{lemma}\label{lem:lemma_rhs}
Suppose that $A \in \Copula$ allows a continuous Markov kernel $K_A$. 
Then $(\CB_{N}(A))_{N \in \mathbb{N}}$ converges uniformly conditional to $A$ for 
$N \rightarrow \infty$, i.e., 
\begin{equation}\label{eq:rhs}
    \lim_{N \rightarrow \infty} \sup_{(x,y) \in \I^2} \abs*{ K_{\CB_{N}(A)}(x, [0, y]) - K_A(x, [0, y]) } =0.
\end{equation}
\end{lemma}
\begin{proof}
Let $\epsilon > 0$ be arbitrary but fixed. By uniform continuity of the function $(x, y) \mapsto K_A(x, [0, y])$ on $\I^2$, there exists some $\delta > 0$, such that for every $(x_1, y_1), (x_2, y_2) \in \I^2$ the following implication holds: 
\[ \abs{x_1 - x_2} + \abs{y_1 - y_2} < \delta \implies \abs{K_A(x_1, [0, y_1]) - K_A(x_2, [0, y_2])} < \tfrac{\varepsilon}{2}\]

\noindent Choose $N_0$ sufficiently large, so that $2/N < \delta$ for every $N \geq N_0$. Consider $N \geq N_0$ and let $(x,y) \in \I^2$ be arbitrary but fixed.  \\
We first derive an alternative version of eq.~\eqref{eq:checkerboard_approx_conditional_df_2}. 
By the mean value theorem we can find $\xi_1, \xi_2 \in \rb*{\tfrac{i_0^N(x) - 1}{N}, \tfrac{i_0^N(x)}{N}}$, such that  
\begin{align*}
    \tfrac{1}{N} K_A \rb*{ \xi_1, \cb*{ 0, \tfrac{j_0^N(y) - 1}{N}}} &= 
     A\rb*{\tfrac{i_0^N(x)}{N}, \tfrac{j_0^N(y) - 1}{N}} 
    - A\rb*{\tfrac{i_0^N(x) - 1}{N}, \tfrac{j_0^N(y) - 1}{N}} \\
    \tfrac{1}{N} K_A \rb*{ \xi_2, \cb*{ 0, \tfrac{j_0^N(y)}{N}}} &= A\rb*{\tfrac{i_0^N(x)}{N}, \tfrac{j_0^N(y)}{N}} - A\rb*{\tfrac{i_0^N(x) - 1}{N}, \tfrac{j_0^N(y)}{N}}
\end{align*} 
hold. Combining the two expressions directly yields 
\begin{align*}
    K_{\CB_N(A)}(x, [0, y]) &= 
    (j_0^N(y) - yN) K_A \rb*{ \xi_1, \cb*{ 0, \tfrac{j_0^N(y) - 1}{N}}} \\
    & \quad + \,(yN + 1 - j_0^N(y)) K_A \rb*{ \xi_2, \cb*{ 0, \tfrac{j_0^N(y)}{N}}}.
\end{align*}
Using the triangle inequality for $\Delta_N:= \abs*{K_{\CB_N(A)}(x, [0, y]) - K_A(x, [0, y])}$ we therefore obtain
\begin{align*}
    \Delta_N \leq \begin{multlined}[t]
        (j_0^N(y) - yN) \abs*{ K_A \rb*{ \xi_1, \cb*{ 0, \tfrac{j_0^N(y) - 1}{N}}} - K_A(x, [0, y])} \\        
        + (yN + 1 - j_0^N(y)) \abs*{ K_A \rb*{ \xi_2, \cb*{ 0, \tfrac{j_0^N(y)}{N}}} - K_A(x, [0, y])}.
    \end{multlined} 
\end{align*}
Considering $\abs*{\xi_1 - x} + \abs*{ \tfrac{j_0^N(y) - 1}{N} - y} < \delta$ and $\abs*{\xi_2 - x} + \abs*{ \tfrac{j_0^N(y)}{N} - y } < \delta$, this implies 
\[ \Delta_N=\abs*{K_{\CB_N(A)}(x, [0, y]) - K_A(x, [0, y])} < \epsilon. \]
Since $(x, y)$ was arbitrary and $N_0$ only depends on $\varepsilon$, the proof is complete. 
\end{proof}

The next lemma shows that in case $A$ allows a continuous Markov kernel, 
the empirical checkerboard approximation $\CB_{N(n)}(E_n)$ 
converges uniformly condition to $A$. 

\begin{lemma}\label{lem:checkerboard_convergence}
Suppose that $(X_1, Y_1), (X_2, Y_2), \dotsc $ is a random sample from $(X,Y)$ and assume 
that $(X,Y)$ has continuous distribution function $H$, mar\-gi\-nals $F$ and $G$, and underlying (unique) copula $A$. Furthermore suppose that $A$ allows a continuous Markov kernel $K_A$.
Setting $N(n) := \floorfct{n^s}$ for some $s \in (0, \frac{1}{2})$, we have 
\begin{equation}\label{eq:checkerboard_convergence}
    \lim_{n \rightarrow \infty} \sup_{(x, y) \in \I^2} \abs*{K_{\CB_{N(n)}(E_n)}(x, [0, y]) - K_A(x, [0,y])} = 0
\end{equation}
almost surely. In other words, the sequence $(\CB_{N(n)}(E_n))_{n \in \mathbb{N}}$ converges uniformly conditional to $A$ almost surely. 
\end{lemma}
\begin{proof}
Setting $\Delta_n(x,y):=\abs*{K_{\CB_{N(n)}(E_n)}(x, [0, y]) - K_A(x, [0,y])}$, 
applying the triangle inequality yields 
\begin{align*}
    \sup_{(x, y) \in \I^2} \Delta_n(x,y)
    &\leq
        \sup_{(x, y) \in \I^2} \abs*{K_{\CB_{N(n)}(E_n)}(x, [0, y]) - K_{\CB_{N(n)}(A)}(x, [0,y])} \\
       & \quad  + \sup_{(x, y) \in \I^2} \abs*{K_{\CB_{N(n)}(A)}(x, [0, y]) - K_A(x, [0,y])}
    \end{align*}
According to Lemma~\ref{lem:lemma_rhs}, the latter summand converges to zero for 
$n \rightarrow \infty$, it thus suffices to show that the former summand converges to zero with probability one.
Using ineq.~\eqref{ineq:check} directly yields
\[ \abs*{K_{\CB_{N(n)}(E_n)}(x, [0, y]) - K_{\CB_{N(n)}(A)}(x, [0,y])} \leq 2 N(n) d_{\infty}(E_n, A) \]
for all $(x,y) \in \I^2$. Since the upper bound does not depend on $(x,y)$ and 
$(x,y) \in \I^2$ was arbitrary we have shown
\[ \sup_{(x, y) \in \I^2} \abs*{K_{\CB_{N(n)}(E_n)}(x, [0, y]) - K_{\CB_{N(n)}(A)}(x, [0,y])} \leq 2 N(n) d_{\infty}(E_n, A). \]
The desired result now follows from our choice of $s$ and the fact that according to \citet{janssen12} 
$$
d_\infty(E_n,A)=O\left(\sqrt{\frac{\log(\log(n))}{n}}\right)
$$
with probability one for $n \rightarrow \infty$.
\end{proof}

\noindent We now have all the tools ready for proving Theorem~\ref{thm:uniform_convergence}.
\begin{proof}[Proof of Theorem~\ref{thm:uniform_convergence}]
Setting  
$$
\Delta_n(x,y):=  \sup_{(x, y) \in \R^2} \abs*{K_{\CB_{N(n)}(E_n)}(F_n(x), [0, G_n(y)]) - 
       K_H(x, (-\infty, y])}
$$
and using eq.~\eqref{eq:sklar.for.kernels} and the triangle inequality directly yields
\begin{align*}
    \Delta_n(x,y)&= \sup_{(x, y) \in \R^2} \abs*{K_{\CB_{N(n)}(E_n)}(F_n(x), [0, G_n(y)]) - K_A(F(x), [0, G(y)])} \\
    &\leq \sup_{(x, y) \in \R^2} \abs*{K_{\CB_{N}(E_n)}(F_n(x), [0, G_n(y)]) - K_A(F_n(x), [0, G_n(y)])} \\
        & \quad + \sup_{(x, y) \in \R^2} \abs*{K_A(F_n(x), [0, G_n(y)]) - K_A(F(x), [0, G(y)])}.
\end{align*}
According to Lemma~\ref{lem:checkerboard_convergence}, the former summand converges to zero almost surely as $n \to \infty$. 
Moreover, using uniform continuity of the function $(x, y) \mapsto K_A(x, [0, y])$ on $\I^2$ 
together with the Glivenko-Cantelli theorem, the latter summand converges to zero almost surely for $n \to \infty$ as well, so the proof is complete. 
\end{proof}

\begin{remark}
\emph{
    To the best of the authors' knowledge, the asymptotics of the empirical checkerboard approximation $\CB_{N(n)}(E_n)$ are still unknown. The preprint \cite{CBEasymptotics} 
    seems to be the first providing results on the asymptotic theory under the 
    strong assumption that the third-order partial derivatives of the 
    underlying copula $C$ are bounded (which are far more restrictive than 
    asking for $C$ to allow a continuous Markov kernel).  
    }
\end{remark}

In the remainder of this section, we show that for copulas that do not admit a continuous Markov kernel, even the checkerboard approximations $\CB_{N}(A)$ fail to converge in the sense of Lemma~\ref{lem:lemma_rhs}.  
Before establishing this result, we present a simple example to illustrate what can go wrong.

\begin{example}
\emph{
For the minimum copula $M$ a version of the Markov kernel $K_M(\cdot,\cdot)$ is given by
\[ K_M(x, [0,y]) = \ind{[0, y]}(x), \quad (x, y) \in \I^2. \]
Obviously the map $(x, y) \mapsto K_M(x, [0, y])$ has a discontinuity at all points of the 
form $(x,x)$ with $x \in \I$. 
According to eq.~\eqref{eq:checkerboard_approx_conditional_df}, a Markov kernel of 
the $N$-checkerboard approximation $\CB_{N(n)}(M)$ of $M$ is given by 
\begin{equation*}
    K_{\CB_{N(n)}(M)}(x, [0,y]) = \begin{cases}
        0, & \text{if $y \leq \tfrac{i_0^N(x) - 1}{N}$} \\
        yN - i_0^N(x) + 1, & \text{if $y \in \rb*{ \tfrac{i_0^N(x) - 1}{N}, \tfrac{i_0^N(x)}{N} } $} \\
        1, & \text{if $y \geq \tfrac{i_0^N(x)}{N}$}
        \end{cases}
\end{equation*}
for every $(x, y) \in \I^2$. Considering $K_M(x,\{x\})=1$ for every $x \in \I$, 
it follows immediately that $K_M(x, [0,y]) \in \curbr{0, 1}$. On 
the other hand, continuity of $y \mapsto K_{\CB_{N(n)}(M)}(x, [0,y])$ implies the existence of some $y_x$ with $K_{\CB_{N}(M)}(x, [0, y_x]) = \tfrac{1}{2}.$
Altogether this shows that for every $x \in  \I$ we even have 
\[ \sup_{y \in \I} \abs*{ K_M(x, [0,y]) - K_{\CB_{N}(M)}(x, [0, y]) } \geq \tfrac{1}{2}. \]
}
\end{example}

We now turn to the general case of copulas not allowing a continuous Markov kernel and 
show that we can not have uniform conditional convergence.
\begin{proposition}
Let $A \in \Copula$ be a copula and $K_A$ be a version of the Markov kernel, such that the function $(x, y) \mapsto K_A(x, [0, y])$ has a discontinuity. Then, the sequence 
$(\Delta_N)_{N \in \N}$, given by
\[ \Delta_N := \sup_{(x, y) \in \I^2} \abs*{K_{\CB_{N}(A)}(x, [0, y]) - K_A(x, [0,y])} \]
does not converge to $0$ for $N \rightarrow \infty$.
\end{proposition}
\begin{proof}
Suppose that $(x_0, y_0) \in \I^2$ is a discontinuity point of the function $(x, y) \mapsto K_A(x, [0, y])$. Then there exists some $\epsilon_0 > 0$ such that for every 
$\delta > 0$ we can find some $(x_{\delta}, y_{\delta}) \in B_\delta(x_0, y_0)$ satisfying
\[ \abs*{K_A(x_{\delta}, [0, y_{\delta}]) - K_A(x_0, [0, y_0])} \geq \epsilon_0. \]
Fix some $N \in \N$ with $x_0, y_0 \notin \curbr*{\tfrac{1}{N}, \dotsc, \tfrac{N-1}{N}}$. 
Then $(x_0,y_0)$ is an inner point of $Q_{i_0^N(x_0), j_0^N(y_0)}^N$, so choosing  
$\delta_N > 0$ sufficiently small we obviously have $B_{\delta_N}(x_0, y_0) \subset Q_{i_0^N(x_0), j_0^N(y_0)}^N$. \\
The function $f: (x, y) \mapsto K_{\CB_N(A)}(x, [0, y])$ is Lipschitz continuous on the interior of $Q_{i_0^N(x_0), j_0^N(y_0)}^N$ (in fact, it is constant in $x$ for given $y$ 
and affine in $y$ given $x$). In particular, $f$ is also Lipschitz continuous on 
$B_{\delta_{\delta'_{N}}}(x_0, y_0)$, hence there exists some $\delta'_N \in (0,\delta_N) $ 
such that for every $(x, y) \in \I^2$ with $\abs{x-x_0} + \abs{y - y_0} < \delta'_N$, we have that the inequality
\[ \abs*{ K_{\CB_N(A)}(x, [0, y]) - K_{\CB_N(A)}(x_0 [0, y_0]) } < \tfrac{\epsilon_0}{2} \]
holds. 
Using the reverse triangle inequality (in the last step) we conclude that
\begin{align*}
    2 \Delta_N &\geq 2 \sup_{(x,y) \in B_{\delta'_N} \rb*{x_0, y_0}} \abs*{K_{\CB_N(A)}(x, [0, y]) - K_A(x, [0, y])} \\
    &\geq \abs*{K_{\CB_N(A)}\rb*{x_{\delta'_N}, [0, y_{\delta'_N}]} - K_A \rb*{x_{\delta'_{N}}, [0, y_{\delta'_{N}}]}} \\
    & \qquad + \abs*{K_{\CB_N(A)}(x_0, [0, y_0]) - K_A(x_0, [0, y_0])} \\
    &\geq \abs*{K_A\rb*{x_{\delta'_N}, [0, y_{\delta'_N}]} - K_A(x_0, [0, y_0])} \\
    &\qquad\qquad- \abs*{K_{\CB_N(A)}\rb*{x_{\delta'_N}, [0, y_{\delta'_N}]} - K_{\CB_N(A)}(x_0, [0, y_0])} \\
    &> \tfrac{\epsilon_0}{2}.
\end{align*}
In other words, we have shown that
\[ \sup_{(x,y) \in \I^2} \abs*{K_{\CB_N(A)}(x, [0, y]) - K_A(x, [0, y])} \geq \epsilon_0/4 \]
holds for every $N \in \N$, such that $x_0, y_0 \notin \curbr*{\tfrac{1}{N}, \dotsc, \tfrac{N-1}{N}}$. Since there are infinitely many such $N$, we are done.
\end{proof}

\section{Uniform Conditional Convergence of the empirical Bernstein approximations}\label{sec:uniform_convergence_bernstein}

Analogous to the empirical checkerboard approximation, in the sequel we will refer to 
the Bernstein approximation of the empirical copula $\B_N(E_n)$ as 
\emph{empirical Bernstein approximation}. 
In a nutshell, in this section we show that the results derived for empirical 
checkerboard approximations also hold for empirical Bernstein approximations (with adequately chosen degree $N=N(s)$).  

\begin{lemma}\label{lem:bernstein_convergence}
Suppose $A \in \Copula$ allows a continuous Markov kernel. Then, we have that
\[ \lim_{N \rightarrow \infty} 
\sup_{(x, y) \in \I^2} \abs*{ K_{\B_N(A)}(x, [0, y]) - K_A(x, [0, y]) }= 0, \]
i.e., $(\B_{N}(A))_{N \in \mathbb{N}}$ converges uniformly conditional to $A$.  
\end{lemma}
\begin{proof}
Let $\epsilon > 0$ be arbitrary but fixed. By uniform continuity of the function $(x, y) \mapsto K_A(x, [0, y])$ on $\I^2$, there exists some $\delta > 0$, such that for every $(x_1, y_1), (x_2, y_2) \in \I^2$ the following implication holds: 
\[ \abs{x_1 - x_2} + \abs{y_1 - y_2} < \delta \implies \abs{K_A(x_1, [0, y_1]) - K_A(x_2, [0, y_2])} < \tfrac{\varepsilon}{2}\]
Choose $N_0$ sufficiently large so that
\[ 2 \exp \rb*{ - 2N^{1-2r} } < \tfrac{\epsilon}{2} \quad \text{and} \quad \tfrac{3 + 2 N^{1-r}}{N} < \tfrac{\delta}{2} \]
hold for every $N \geq N_0$, whereby $r \in \rb*{0, \tfrac{1}{2}}$. \\
Using the mean value theorem, for every $i \in \curbr*{1, \dotsc, N}$ and 
$j \in \curbr*{0, \dotsc, N}$ we can find some $\xi_{i, j} \in \rb*{ \tfrac{i - 1}{N}, \tfrac{i}{N} }$ such that
\[ \tfrac{1}{N} K_A \rb*{ \xi_{i, j} \cb*{0, \tfrac{j}{N}} } = A \rb*{ \tfrac{i}{N}, \tfrac{j}{N} } - A \rb*{ \tfrac{i - 1}{N}, \tfrac{j}{N} }. \]
This shows that the last line in eq.~\eqref{eq:markov_kernel_bernstein} can be rewritten as
\begin{align*}
    K_{\B_N(A)}(x, [0, y])
    = \sum_{i, j = 1}^N K_A\rb*{ \xi_{i, j}, \cb*{ 0, \tfrac{j}{N} } } p_{N-1, i-1}(x) p_{N, j}(y).
\end{align*}
Setting $\Delta_N(x,y):=\abs*{ K_{\B_N(A)}(x, [0, y]) - K_A(x, [0, y]) }$ 
for every $(x,y) \in \I^2$ and using the fact that Bernstein polynomials are a partition of unity, we have
\begin{multline}\label{eq:bernstein_copula_ineq}
    \Delta_N(x,y) \leq
    \sum_{i=1}^N \sum_{j=1}^N p_{N-1, i-1}(x) p_{N, j}(y)  \abs*{ K_A \rb*{ \xi_{i, j}, \cb*{0, \tfrac{j}{N} } } - K_A(x, [0, y]) }.
\end{multline}
Defining the sets $U_x, V_y \subset \curbr{0, \dotsc, N}$ by
\begin{align*}
    U_x &:=  \curbr*{ 0, \dotsc, \floorfct{x(N-1) - \Delta_N}, \ceilfct{x(N-1) + \Delta_N}, \dotsc, N-1 } \\
    V_y &:=  \curbr*{ 0, \dotsc, \floorfct{yN - \Delta_N}, \ceilfct{yN + \Delta_N}, \dotsc, N }.
\end{align*}
the right-hand side of eq.~\eqref{eq:bernstein_copula_ineq} can be expressed as
\[ \abs*{ K_{\B_N(A)}(x, [0, y]) - K_A(x, [0, y]) } \leq S_1(x, y) + S_2(x, y), \]
with
\begin{align*}
    S_1(x, y) &= \sum_{i \in U_x \lor j \in V_y} p_{N-1, i-1}(x) p_{N, j}(y)  \abs*{ K_A \rb*{ \xi_{i, j}, \cb*{0, \tfrac{j}{N} } } - K_A(x, [0, y]) } \\
    S_2(x, y) &= \sum_{i \notin U_x, j \notin V_y} p_{N-1, i-1}(x) p_{N, j}(y)  \abs*{ K_A \rb*{ \xi_{i, j}, \cb*{0, \tfrac{j}{N} } } - K_A(x, [0, y]) }.
\end{align*}
By Lemma~\ref{lem:bernstein_chernoff} we have that 
$S_1(x, y) \leq 2 \exp \rb*{ - 2 N^{1-2r} } < \tfrac{\epsilon}{2} $ for every $(x, y) \in \I^2$ and $N \geq N_0$. 
The sum in $S_2(x, y)$ only contains indices satisfying $\abs*{ \xi_{i, j} - x } + \abs*{ y - \tfrac{j}{N}} < \delta$, implying $S_2(x, y) < \tfrac{\epsilon}{2}$. 
Altogether this shows that for every $N \geq N_0$ we have
\[ \abs*{ K_{\B_N(A)} ( x, [0, y] ) - K_A(x, [0, y] } < \epsilon. \]
Since $(x, y)$ was arbitrary and $N_0$ only depends on $\varepsilon$, the proof is complete.
\end{proof}

The next lemma is the Bernstein version of Lemma~\ref{lem:checkerboard_convergence}.
We omit the proof since it is analogous to the one of  Lemma~\ref{lem:checkerboard_convergence}, we just use ineq.~\eqref{ineq:bern} 
instead of ineq.~\eqref{ineq:check}.

\begin{lemma}
Let $A \in \Copula$ be a copula allowing a continuous Markov kernel $K_A$. 
Setting $N(n) := \floorfct{n^s}$ for some $s \in \rb*{0, \tfrac{1}{2}}$, we have that
\[ \lim_{n \rightarrow \infty}\sup_{(x, y) \in \I^2} \abs*{ K_{\B_{N(n)}(E_n)}(x, [0, y]) - K_A(x, [0, y]) }= 0 \]
almost surely. In other words: the sequence $(\B_{N(n)}(E_n))_{n \in \mathbb{N}}$ 
converges uniformly conditional to $A$ almost surely.  
\end{lemma}
%
%

We conclude this section by stating its main result, whose proof again 
is analogous to the one of Theorem~\ref{thm:uniform_convergence}. 
\begin{theorem}\label{thm:uniform_convergence_bernstein}
Suppose that $(X_1, Y_1), (X_2, Y_2), \dotsc$ is a random sample from $(X,Y)$ and assume 
that $(X,Y)$ has continuous distribution function $H$, mar\-gi\-nals $F$ and $G$, and underlying (unique) copula $A$. Furthermore suppose that $A$ allows a continuous Markov kernel $K_A$ and define $K_H$ according to eq.~\eqref{eq:sklar.for.kernels}. 
Setting $N(n) := \floorfct{n^s}$ for some $s \in \rb*{0, \frac{1}{2}}$, we have that
\begin{equation}\label{eq:uniform_convergence_bernstein}
    \lim_{n \rightarrow \infty}  \sup_{(x, y) \in \R^2} \abs*{K_{\B_{N(n)}(E_n)}(F_n(x), [0, G_n(y)]) - K_H(x, (-\infty, y])} =0
\end{equation}
almost surely. 
\end{theorem}

Considering that for every copula $A$ and every $N \in \N$ the function 
$(x,y) \mapsto K_{\B_{N}(A)}(x,[0,y])$ is continuous the subsequent proposition 
follows immediately. 
\begin{proposition}
Let $A \in \Copula$ be a copula and $K_A$ be a version of the Markov kernel, such that the function $(x, y) \mapsto K_A(x, [0, y])$ has a discontinuity. Then, the sequence 
$(\Delta_N)_{N \in \N}$, given by
\[ \Delta_N := \sup_{(x, y) \in \I^2} \abs*{K_{\B_{N}(A)}(x, [0, y]) - K_A(x, [0,y])} \]
does not converge to $0$ for $N \rightarrow \infty$.
\end{proposition}

\section{Consequences to Nonparametric Regression}\label{sec:regression}

Under the mild regularity condition that the copula $A$ allows a continuous version of the Markov kernel, Theorem~\ref{thm:uniform_convergence} assures not only weak but even 
uniform convergence of the estimated conditional distribution functions to the true ones. 
Not surprisingly, this has direct consequences to functionals of the 
conditional distribution, 
including in particular mean, quantile, and expectile regression as well as
conditional variance and beyond. 
In the sequel we will focus on nonparametric regression, the obtained results on mean and quantile regression are an extension of the bivariate results in \citet{kaiser25} to the general distributional case. 
Doing so we will only work with the empirical checkerboard approximations, 
using Theorem~\ref{thm:uniform_convergence_bernstein} all results, however, also hold for the empirical Bernstein approximation.

For the remainder of this section we consider a pair $(X,Y)$ of random variables 
with continuous distribution function $H$ and let $K_H$ denote the Markov kernel according to eq. \eqref{eq:sklar.for.kernels}. 

\subsection{Mean regression}\label{subsec:mean_regression}
Recall that the conditional expectation can be calculated in terms of integrals of 
the conditional distribution function. In fact,  
assuming that $Z$ having distribution $K_H(x,\cdot)$ is integrable we have  
\begin{align*}
r_H(x)&:= \E{ Y|X = x } = \int_{(0,\infty)} 1-K_H(x, (-\infty,t]) \dd{\lambda(t)}  \\
  & \quad - \int_{(-\infty,0)} K_H(x, (-\infty, t] \dd{\lambda(t)}. 
\end{align*}
Theorem~\ref{thm:uniform_convergence} implies that we can consistently estimate 
the regression function without further ado as long as it is possible to 
interchange limit and 
integral. The following technical condition guarantees validity of the interchange.
\begin{assumption}\label{def:reg}
For $x \in \mathbb{R}$ let the random variable $Z_n^{(x)}$ have distribution function $K_{\CB_{N(n)}(E_n)}(F_n(x), [0, G_n(\cdot)])$. We say that the regularity 
assumption holds in a point $x \in \mathbb{R}$ if the sequence 
$(Z_n^{(x)})_{n \in \mathbb{N}}$ is uniformly integrable. 
\end{assumption}

The following remark shows that uniform integrability is not too strict and 
holds, e.g., in the case of uniform boundedness.
\begin{remark}\label{rmk:UI}
\emph{
Suppose that $Z,Z_1,Z_2,\ldots$ are random variables with 
$\vert Z_n \vert \leq Z$ almost surely for every $n \in \N$. 
Then obviously the sequence $(Z_n)_{n \in \N}$ is uniformly integrable if $Z$ is integrable. 
Returning to our setting of $(X,Y)$ with continuous distribution function $H$, 
if the random variable $Y$ is bounded, 
then we have that $(Z_n^{(x)})_{n \in \mathbb{N}}$ is uniformly integrable, which can 
be shown
as follows: Assume that $\vert Y \vert \leq a$ almost surely and that 
$(X_1,Y_1),(X_2,Y_2),\ldots$ is a random sample from $(X,Y)$. 
Then obviously we have 
$G_n(a)=G(a)=1$ as well as $G_n(-a)=G(-a)=0$ almost surely for every $n \in \mathbb{N}$, implying that for every $x \in \mathbb{R}$ 
\begin{align*}
K_{\CB_{N(n)}(E_n)}(F_n(x), [0, G_n(a)])&=K_{\CB_{N(n)}(E_n)}(F_n(x), [0,1])=1  \\ 
K_{\CB_{N(n)}(E_n)}(F_n(x), [0, G_n(-a)])&=K_{\CB_{N(n)}(E_n)}(F_n(x), \{0\})=0    
\end{align*}
hold almost surely. As a direct consequence, for each $x \in \mathbb{R}$
the sequence $(Z_n^{(x)})_{n \in \mathbb{N}}$ is bounded from above by $a$ 
and from below by $-a$, and therefore uniformly integrable.} 
\end{remark}

\begin{theorem}\label{cor:mean_reg}
Suppose that $(X_1, Y_1), (X_2, Y_2), \dotsc$ is a random sample from $(X,Y)$ and assume 
that $(X,Y)$ has continuous distribution function $H$, mar\-gi\-nals $F$ and $G$, and underlying (unique) copula $A$. Furthermore suppose that $A$ allows a continuous Markov kernel $K_A$, define $K_H$ according to eq.~\eqref{eq:sklar.for.kernels} and set $N(n) := \floorfct{n^s}$ for some $s \in \rb*{0, \frac{1}{2}}$. 
Then for every $x \in \R$ in which the regularity assumption holds we have 
\begin{equation*}
    r_n(x) := r_n^+(x) - r_n^-(x) \to \E{ Y \mid X = x}
\end{equation*}
almost surely as $n \to \infty$, where
\begin{align*}
    r_n^+(x) &:= \int_{(0,\infty)} 1 - K_{\CB_{N(n)}(E_n)}(F_n(x), [0, G_n(y)]) \dd{\lambda(y)}, \\
    r_n^-(x) &:= \int_{(-\infty,0)} K_{\CB_{N(n)}(E_n)}(F_n(x), [0, G_n(y)]) \dd{\lambda(y)}.
\end{align*}
\end{theorem}
\begin{proof}
Direct consequence of Theorem \ref{thm:uniform_convergence} and the fact that 
uniform integrability allows to interchange limit and integral. 
\end{proof}

Since the empirical distribution functions $F_n,G_n$ are step functions, the 
two integrals in Theorem \ref{cor:mean_reg} boil down to sums as the following remark 
shows.
\begin{remark}\label{rmk:mean_reg_big_o}
\emph{
Let $Y^{(k)}$ denote the $k$-th order statistic and consider $Y_+^{(k)} := \max\curbr{0, Y^{(k)}}$ and $Y_-^{(k)} := \max\curbr{0, -Y^{(k)}}$. 
Then the two integrals in Theorem \ref{cor:mean_reg} calculate to 
\begin{align*}
  r_n^+(x) &= \sum_{i=1}^{n-1} \rb*{Y_+^{(i+1)} - Y_+^{(i)}} \rb*{ 1 - K_{\CB_{N(n)}(E_n)} \rb*{F_n(x), \cb*{0, G_n\rb*{Y^{(i)}}}} } \nonumber \\
  &= \sum_{i=1}^{n-1} \rb*{Y_+^{(i+1)} - Y_+^{(i)}} \rb*{ 1 - K_{\CB_{N(n)}(E_n)} \rb*{F_n(x), \cb*{0, \tfrac{i}{n} }}} 
\end{align*}  
as well as 
\begin{align*}
 r_n^-(x) &= \sum_{i=1}^{n-1} \rb*{Y_-^{(i)} - Y_-^{(i+1)}} \rb*{ K_{\CB_{N(n)}(E_n)} \rb*{F_n(x), \cb*{0, G_n\rb*{Y^{(i)}}}} } \\
  &= \sum_{i=1}^{n-1} \rb*{Y_-^{(i)} - Y_-^{(i+1)}} \rb*{ K_{\CB_{N(n)}(E_n)} \rb*{F_n(x), 
\cb*{0, \tfrac{i}{n} }} } 
\end{align*}
Considering that $r_n$ is a step function attaining at most $N(n)$ different values 
allows us to compute the estimated mean regression very efficiently. 
Indeed, we can compute the $N(n)$ different values in $\bigO{N(n)n}$. 
In order to evaluate $m \in \N$ points therefore requires a time complexity of 
$\bigO{m + N(n)n}$. Other standard well-known methods for nonparametric 
mean regression like, e.g., the Nadaraya Watson estimator needs 
$\bigO{mn}$ to evaluate $m$ points. 
This implies that in case of $m \geq N(n)$ our proposed method is computationally 
much faster.}
\end{remark}


\subsection{Quantile regression}\label{subsec:quantile_regression}
Next we focus on quantile regression and utilize the empirical checkerboard approximation 
to nonparametrically estimate quantile regression functions. 
For $x \in \R$ and quantile level $\tau \in (0, 1)$ define the (interval-valued)  
$\tau$-quantile of a univariate distribution function $F$ according to \citep{jordan22} by
\begin{equation}
\begin{aligned}
    \cb*{ \underline{q}^\tau, \overline{q}^\tau } := \cb*{ \sup \curbr*{ y \in \R: F(y) < \tau }, \inf \curbr*{ y \in \R: F(y) > \tau } }.
\end{aligned}
\end{equation}
Notice that $\cb*{ \underline{q}^\tau, \overline{q}^\tau }$ is an interval with positive length if, and only if there exist points $y_1<y_2$ such that $F(y)=\tau$ for every 
$y \in (y_1,y_2)$.
The following simple lemma will be used in the sequel: 
\begin{lemma}\label{lem:quant_reg}
Suppose that $F, F_1, F_2, \dotsc$ are univariate distribution functions, 
such that $\rb*{F_n}_{n \in \N}$ converges weakly to $F$ and denote the 
corresponding $\tau$-quantiles by  
$\cb*{ \underline{q}^\tau, \overline{q}^\tau }$, $\cb*{ \underline{q}_1^\tau, \overline{q}_1^\tau },\cb*{ \underline{q}_2^\tau, \overline{q}_2^\tau }\ldots$.  
Then we have 
\[ \lim_{n \rightarrow \infty}\, \max_{ y_n \in \cb*{ \underline{q}_n^\tau, \overline{q}_n^\tau } } \min_{ y \in \cb*{ \underline{q}^\tau, \overline{q}^\tau } } \abs*{ y_n - y } =0. \]
\end{lemma}

\begin{proof}
Let $\epsilon > 0$ be arbitrary but fixed. By definition of the $\tau$-quantile 
we obviously have $F\rb*{ \underline{q}^\tau - \frac{\epsilon}{2}} < \tau < F\rb*{ \overline{q}^\tau + \frac{\epsilon}{2} }$. 
Using the fact that weak convergence is equivalent to convergence on a dense subset 
of $\mathbb{R}$ we can find some $n_0 \in \N$, such 
that $F_n\rb*{ \underline{q}^\tau - \epsilon } < \tau < F_n\rb*{ \overline{q}^\tau + \epsilon}$ for every $n \geq n_0$. 
This shows 
\[ \underline{q}^\tau - \epsilon \leq \underline{q}_n^\tau \leq \overline{q}_n^\tau \leq \overline{q}^\tau + \epsilon, \]
as well as
$$
\max_{ y_n \in \cb*{ \underline{q}_n^\tau, \overline{q}_n^\tau } } \min_{ y \in \cb*{ \underline{q}^\tau, \overline{q}^\tau } } \abs*{ y_n - y } \leq \varepsilon
$$
for every $n \geq n_0$. Since $\epsilon$ was arbitrary, the proof is complete. 
\end{proof}
We can now show that the empirical checkerboard approximation 
directly leads to a consistent estimator of the $\tau$-quantile regression function 
$x \mapsto Q_H^\tau(x) =\cb*{\underline{q}^\tau(x), \overline{q}^\tau(x)} $, given by 
\begin{align}\label{eq:lower.upper.quantile}
    \underline{q}^\tau(x)& := \sup \curbr*{ y \in \R: K_H(x,(-\infty,y]) < \tau } \nonumber\\ 
    \overline{q}^\tau(x)& := \inf \curbr*{ y \in \R: K_H(x,(-\infty,y]) > \tau },
\end{align}
where $K_H$ is again defined according to eq. \eqref{eq:sklar.for.kernels}. 
To simplify notation we will write 
$ Q_n^\tau(x) := \cb*{\underline{q}_n^\tau(x), \overline{q}_n^\tau(x)} $
with 
\begin{align}\label{eq:lower.upper.quantile.sample}
    \underline{q}_n^\tau(x)& := \sup \curbr*{ y \in \R: K_{\CB_{N(n)}(E_n)}(F_n(x), [0, G_n(y)]) < \tau } \nonumber \\ 
    \overline{q}_n^\tau(x)& := \inf \curbr*{ y \in \R: K_{\CB_{N(n)}(E_n)}(F_n(x), [0, G_n(y)]) > \tau }.
\end{align}
Combining Theorem~\ref{thm:uniform_convergence} and Lemma \ref{lem:quant_reg} 
directly yields the following result:  
\begin{theorem}\label{cor:quant_reg}
Suppose that $(X_1, Y_1), (X_2, Y_2), \dotsc$ is a random sample from $(X,Y)$ and assume 
that $(X,Y)$ has continuous distribution function $H$, mar\-gi\-nals $F$ and $G$, and underlying (unique) copula $A$. Furthermore suppose that $A$ allows a continuous Markov kernel $K_A$, define $K_H$ according to eq.~\eqref{eq:sklar.for.kernels} and set $N(n) := \floorfct{n^s}$ for some $s \in \rb*{0, \frac{1}{2}}$.  Then, for every 
$x \in \mathbb{R}$ the identity 
\[ \lim_{n \rightarrow \infty}\, \max_{ y_n \in \cb*{ \underline{q}_n^\tau(x), \overline{q}_n^\tau(x) } } \min_{ y \in \cb*{ \underline{q}^\tau(x), \overline{q}^\tau(x) } } \abs*{ y_n - y } =0 \]
holds almost surely. 
\end{theorem}

\subsection{Expectile Regression}

As third and last consequence of Theorem \ref{thm:uniform_convergence} we derive 
a novel method of expectile regression. To the best of our knowledge, there are not many 
established non-parametric expectile regression methods available, our approach may 
therefore be a valuable contribution filling a gap. \\
Consider $x \in \R$ and let $\alpha \in (0, 1)$ be the expectile level. 
Recall that for a random variable $Y$ the $\alpha$-expectile is defined as the 
unique $e_\alpha \in \mathbb{R}$ satisfying
\[ \alpha \E{ (Y - e_\alpha)_+} = (1 - \alpha) \E{ (e_\alpha - Y)_+ }. \]
Returning to our bivariate setting, considering a pair $(X,Y)$ of random variables 
with continuous distribution function $H$, and letting $K_H$ be defined 
according to eq. \eqref{eq:sklar.for.kernels} by definition we have that the conditional 
$\alpha$-expectile $e_\alpha(x)$ of $Y$ given $X=x$ is the unique solution of the
equation 
\[ \alpha \int_{e_\alpha(x)}^{\infty} 1 - K_H(x, \cbr{-\infty, y}) \dd{\lambda(y)} = (1 - \alpha) \int_{-\infty}^{e_\alpha(x)} K_H (x, \cbr{-\infty, y}) \dd{\lambda(y)}. \]

\begin{theorem}
Suppose that $(X_1, Y_1), (X_2, Y_2), \dotsc$ is a random sample from $(X,Y)$ and assume 
that $(X,Y)$ has continuous distribution function $H$, mar\-gi\-nals $F$ and $G$, and underlying (unique) copula $A$. Furthermore suppose that $A$ allows a continuous Markov kernel $K_A$, define $K_H$ according to eq.~\eqref{eq:sklar.for.kernels} and set $N(n) := \floorfct{n^s}$ for some $s \in \rb*{0, \frac{1}{2}}$. 
For $x \in \R$ fulfilling the regularity assumption and every $n \in \mathbb{N}$ 
define $t_n \in \mathbb{R}$ as the solution of 
\begin{multline*}
    \alpha \int_{t_n}^{\infty} 1 - K_{\CB_{N(n)}(E_n)}(F_n(x), \cb{0, G_n(y)}) \dd{\lambda(y)} \\
    = (1 - \alpha) \int_{-\infty}^{t_n} K_{\CB_{N(n)}(E_n)}(F_n(x), \cb{0, G_n(y)} \dd{\lambda(y)}.
\end{multline*}
Then we have that $(t_n)_{n \in \mathbb{N}}$ converges to $e_\alpha(x)$ almost surely. 
\end{theorem}
\begin{proof}
Using the triangle inequality and the assumption that we can exchange the integral and the limit, we have that
\begin{align*}
    &\abs*{ \alpha \int_{t_n}^\infty 1 - K_H(x, (-\infty, y]) \dd{\lambda(y)} - ( 1 - \alpha ) \int_{-\infty}^{t_n} K_H(x, (-\infty, y]) \dd{\lambda(y)} } \\
    &\quad \leq \begin{multlined}[t]
        \alpha \abs*{ \int_{t_n}^\infty K_{\CB_{N(n)}(E_n)}(F_n(x), [0, G_n(y)]) - K_H(x, (-\infty, y]) \dd{\lambda(y)}  } \\
        + (1 - \alpha) \abs*{ \int_{-\infty}^{t_n} K_{\CB_{N(n)}(E_n)}(F_n(x), [0, G_n(y)]) - K_H(x, (-\infty, y]) \dd{\lambda(y)}  }
    \end{multlined} 
\end{align*}
and the latter two sums converge to $0$ almost surely as $n \to \infty$. 
This implies that $t_n \to t$ as $n \to \infty$, since the function
\[ c \mapsto \abs*{ \alpha \int_{c}^\infty 1 - K_H(x, (-\infty, y]) \dd{\lambda(y)} - ( 1 - \alpha ) \int_{-\infty}^{c} K_H(x, (-\infty, y]) \dd{\lambda(y)} } \]
is convex and thus has a unique minimum.
\end{proof}


\section{Simulation Study}\label{sec:simulation}

The purpose of this section is twofold, firstly, to illustrate the results 
derived in the previous sections and, secondly, to compare our novel estimators
in the regression setting with existing ones. 
As before, we first focus on the co\-pula setting and visualize    Lemma~\ref{lem:checkerboard_convergence} in terms of simulations. 
We then proceed to simulations concerning Theorem~\ref{cor:mean_reg} and Theorem~\ref{cor:quant_reg} and close the section with a simulation underlining the fact 
that the empirical checkerboard approximations can also be used to estimate 
other conditional quantities of interest like, e.g., conditional variance. 

\subsection{Copula setting}

As our first example we consider the Ali-Mikhail-Haq (AMH) family, which 
consists of copulas allowing a continuous Markov kernel.

\begin{example}
\emph{
The AMH copula with parameter $\theta \in [-1, 1)$  (see \citet[Example~6.5.19]{sempi15}) is given by
\[ C_\theta(x, y) = \frac{xy}{1 - \theta (1 - x) (1 - y)}. \]
The continuous version of the Markov kernel of $C_{\theta}$ is given by
\[ K_{C_{\theta}}(x, [0, y]) = \frac{\theta(y-1)y + y}{(1 - \theta(1-x)(1-y))^2} \]
for $x,y \in \I$. Considering $\theta=0.75$, drawing samples from $(X,Y) \sim C_\theta$ of 
increasing size 
$n \in \{100,250,2500,10000,25000\}$, evaluating  
$$
\abs*{ K_{C_\theta}(x, [0, y]) - K_{\CB_{N(n)}(E_n)}(x, [0, y])}
$$
on a random sample of size $m = 2 \cdot \floorfct{n^{0.45}}^{2}$ from $(U,V) \sim \Pi$
and calculating the maximal value over these points yields a value $v \in \I$ for each 
sample size $n$. Interpreting this value $v$ as approximation for 
\[ \sup_{(x, y) \in \I^2} \abs*{ K_{C_\theta}(x, [0, y]) - K_{\CB_{N(n)}(E_n)}(x, [0, y])}, \]
and repeating the simulations $R=2.000$ times yields the results gathered in the 
left panel of 
Figure~\ref{fig:AMH_Copula_convergence}. With increasing $n$ the 
errors seem to tend to zero, confirming Lemma~\ref{lem:lemma_rhs}.
}
\end{example}

As second example we consider Clayton copulas, which do not allow a continuous version 
of the Markov kernel and show that in this case the 
empirical checkerboard approximation does not converge uniformly conditional. 
\begin{example}
\emph{
The Clayton copula with parameter $\theta \in [-1, \infty) \setminus\{0\}$ (see  \citet[Example~6.5.17]{sempi15}) is given by
\[ C_\theta(x, y) = \max \curbr*{ \rb*{ x^{-\theta} + y^{-\theta} - 1}^{-1/\theta} , 0}. \]
In this case, for $x,y \in (0,1]$ a version of the Markov kernel is given by
\[ K_{C_{\theta}}(x, [0, y]) =  \max \curbr*{ x^{-\theta - 1}\rb*{ x^{-\theta} + y^{-\theta} - 1}^{-\frac{\theta + 1}{\theta}} , 0}, \]
which is continuous at every $(x,y)\in (0, 1]^2$. It is straightforward to show that $C_\theta$ 
does not allow
a continuous Markov kernel. In fact, considering an arbitrary copula $C$, 
disintegration implies
$$
0=\mu_C(\I \times \{0\}) = \int_\I K_C(x,\{0\}) \dd \lambda(x),
$$
so for $\lambda$-almost every $x \in \I$ every kernel has to fulfil $K_C(x,\{0\})=0$. 
As a direct consequence, every continuous Markov kernel $K_C$ has to fulfill   
$K_C(0,\{\0\})= 0$. Returning to the Clayton copula $C_\theta$, assuming that
$(x,y) \mapsto K_{C_\theta}(x,[0,y])$ would be continuous on $\I^2$ and
using the afore\-men\-tioned pro\-perty yields  
\[ K_{C_\theta}(0, \{0\}) =  K_{C_\theta}(0, [0, 0]) = 0 \neq 2^{-\frac{\theta+1}{\theta}} = \lim_{\epsilon \to 0} K_{C_\theta}(\epsilon, [0, \epsilon]), \]
a contradiction. Hence $C_\theta$ does not allow a continuous Markov kernel.
Considering $\theta = 2$ and the analogous simulation setup as for the AHM family 
we obtain the results depicted in the right panel of Figure~\ref{fig:Clayton_Copula_convergence}. 
The boxplots insinuate that, contrary to the AHM simulation, the maximal distance 
does not tend to zero as $n \rightarrow \infty$.}
\end{example}


\begin{figure}
\centering
\begin{subfigure}[b]{0.485\textwidth}
     \includegraphics[width=\textwidth]{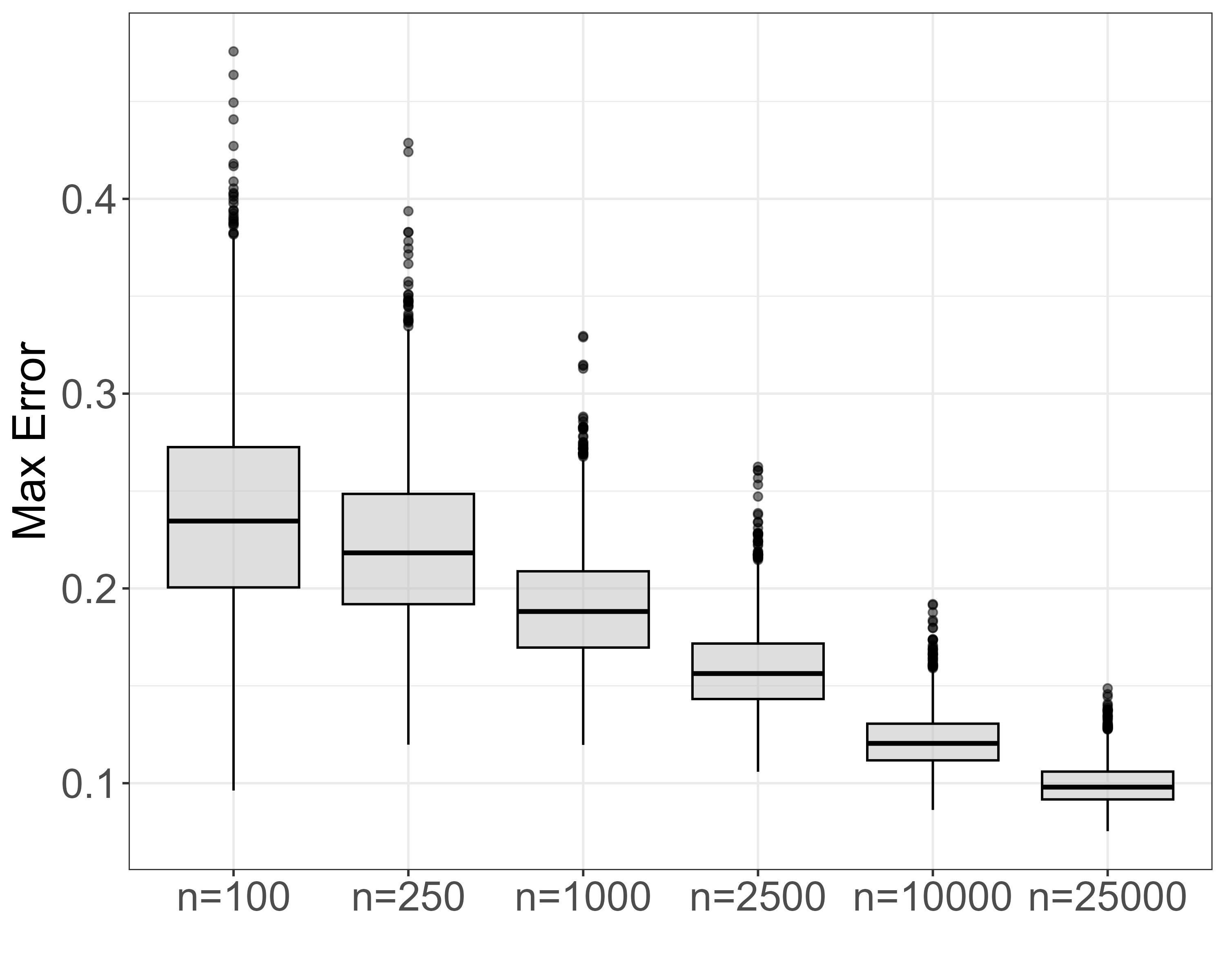}
     \caption{Ali-Mikhail-Haq copula with $\theta = 0.75$}
     \label{fig:AMH_Copula_convergence}
\end{subfigure}
~
\begin{subfigure}[b]{0.485\textwidth}
     \includegraphics[width=\textwidth]{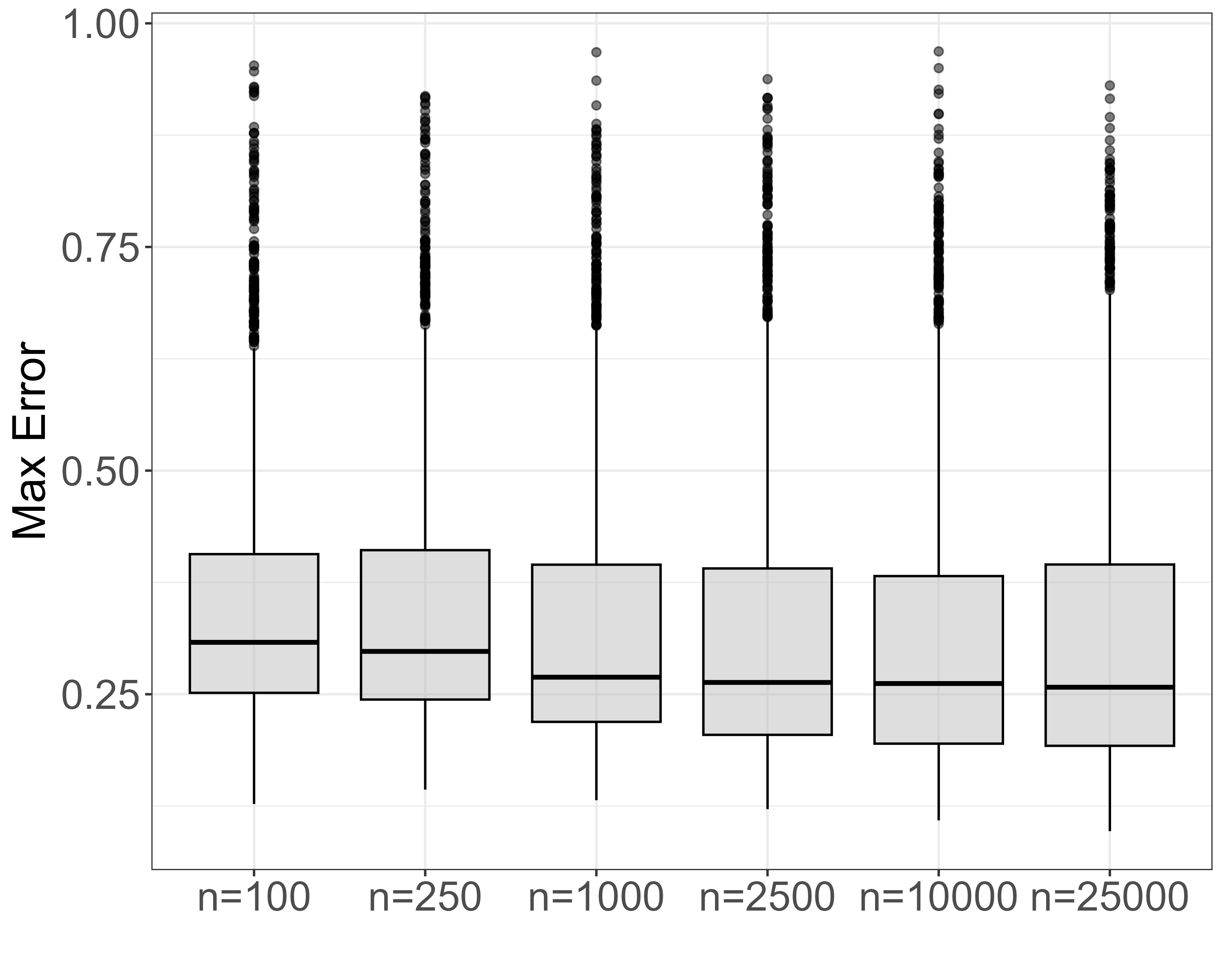}
     \caption{Clayton copula with $\theta = 2$}
     \label{fig:Clayton_Copula_convergence}
\end{subfigure}
\caption{Convergence of the checkerboard approximation for the AMH copula (a) and the Clayton copula (b).}
\label{fig:Copula_convergence}
\end{figure}






\subsection{General regression setting}
We consider the following simulation setup. Let $s, \theta: [0, 10] \to (0,\infty)$ 
be smooth, positive functions, taking the role of the shape and scale of the Gamma distribution, respectively. 
In order to assure that the regularity assumption 
globally holds, we approximate this gamma distribution function as follows: 
Let $c$ be a constant, such that $c \gg s \cdot \theta$ and suppose that 
\begin{align*}
    X &\sim 10 \cdot \operatorname{Beta}(\alpha, \beta), \\
    Y \;|\; X & \sim c \cdot \operatorname{Beta}( \alpha'(X), \beta'(X) ),
\end{align*}
where $\alpha' > 0$ and $\beta' > 0$ are continuous functions with 
$$
\E{ Y \mid X } = s(X) \theta(X), \quad \mathbb{V}( Y \mid X ) =  s(X) \theta(X)^2  
$$
Since $\abs{Y}$ is bounded by $c$, Remark~\ref{rmk:UI} implies that 
the regularity assumptions used in 
Theorem~\ref{cor:mean_reg} and Theorem~\ref{cor:quant_reg} are satisfied.  \\

We now consider mean regression and compare our empirical checkerboard estimator (CBE)
as analyzed in Section~\ref{subsec:mean_regression} with the popular Nadaraya-Watson estimator (NWE).  According to Remark~\ref{rmk:mean_reg_big_o}, the CBE is computationally 
faster than the NWE when facing many predictions. Notice that another advantage of the CBE 
is that there is no need for specifying a hyperparameter for the bandwidth. 
Subsequently, we use Silverman's rule of thumb \citep{silverman98} for selecting 
the bandwidth $h_n$ according to  
$h_n = \widehat{\sigma}_n(X) n^{-\frac{1}{5}}$.

\begin{example}\label{ex:gamma_standard}
\emph{
Suppose that $X \sim \Unif(0, 10)$ and consider  
\[ s(x) =  \max\curbr{0.5, \sqrt{x}}, \quad \text{and} \quad \theta(x) = \min \curbr{ \max \curbr{1, x}, 6}. \]
We chose this example since it helps visualising the performance of the two methods CBE 
and NWE under heteroscedasticity.
\begin{figure}
\centering
\includegraphics[width=\textwidth]{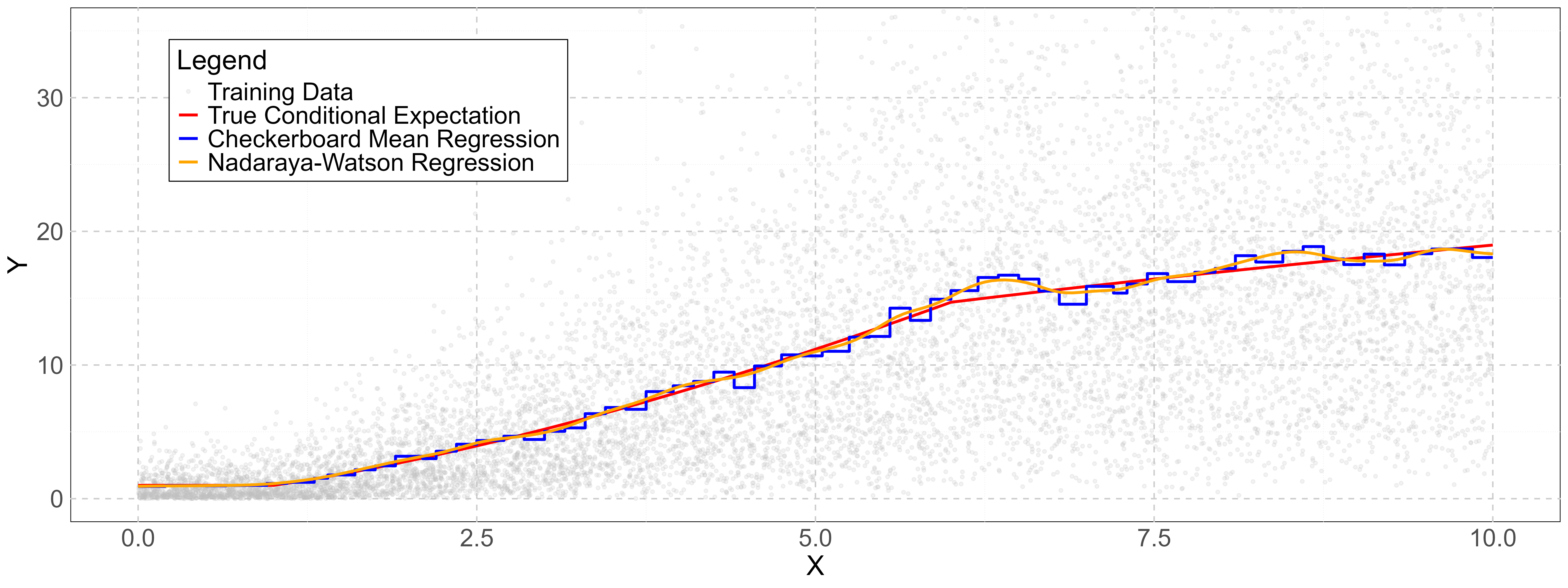}
\caption{Sample of size $n = 10.000$ from the distribution considered in
Example \ref{ex:gamma_standard}. CBE for the regression function (blue), NWE (orange) and true regression 
function (red). }
\label{fig:mean_regression_10000}
\end{figure}
\noindent Figure~\ref{fig:mean_regression_10000} depicts a sample of $n = 10.000$ points 
from this distribution as well as the CBE (blue), the NWE (orange) and 
the true regression function (red). We can clearly see the piecewise constant form of 
the CBE compared to the continuous NWE. The performance of the CBE and the NWE 
is similar in the low-variance region. In the region with high variance, the CBE is 
less smooth. 
\begin{figure}
\centering
\begin{subfigure}[b]{0.9\textwidth}
    \includegraphics[width=\textwidth]{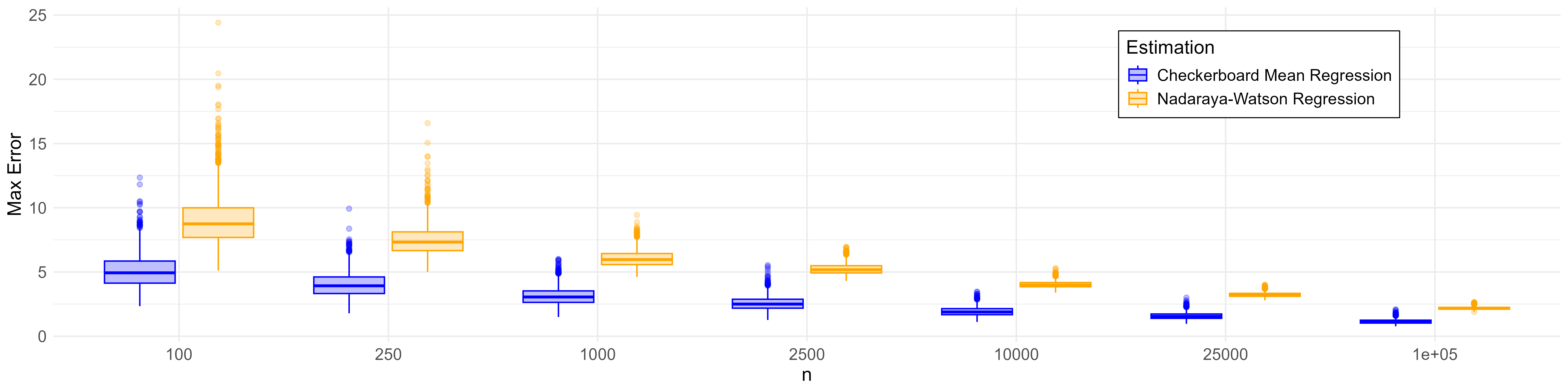}
    \caption{Max Error}
    \label{fig:mean_regression_convergence_max}
\end{subfigure}
~
\begin{subfigure}[b]{0.9\textwidth}
    \includegraphics[width=\textwidth]{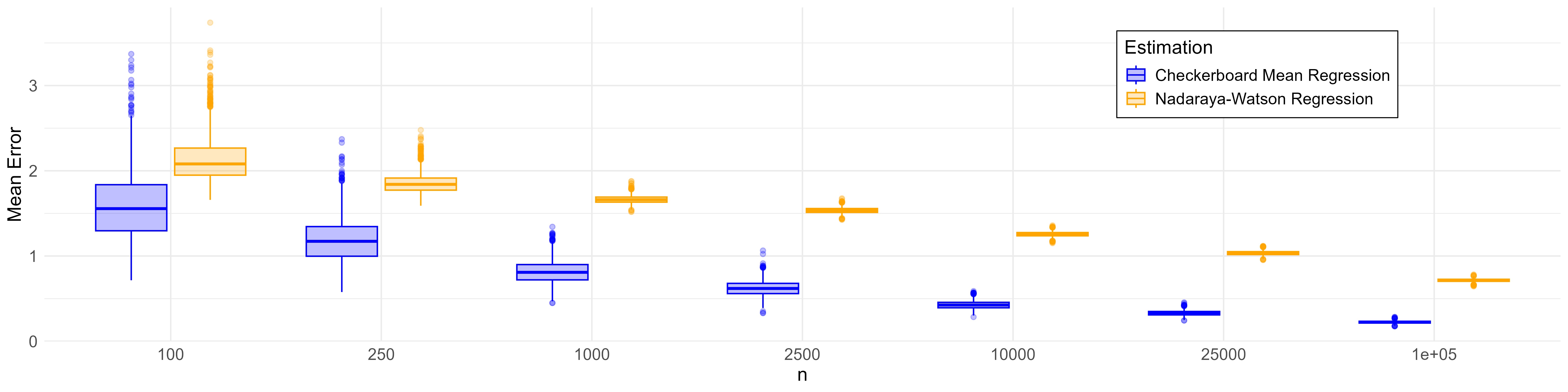}
    \caption{Mean Error}
    \label{fig:mean_regression_convergence_mean}
\end{subfigure}
\caption{(Approximations of the) Mean absolute errors of the CBE (blue) and the 
NWE (orange) for the distribution considered in Example \ref{ex:gamma_standard}; 
the mean/max absolute error was calculated by evaluating at $m = 2.000$ randomly 
generated points.}
\label{fig:mean_regression_convergence}
\end{figure}
Calculating the maximum and the mean absolute error (over a grid of $m = 2.000$ randomly 
generated points drawn from $\mathbb{P}^X$), repeating the procedure $R=2.000$ times and summarizing the 
obtained values in terms of boxplots yields the results  
depicted in Figure~\ref{fig:mean_regression_convergence_max} and Figure~\ref{fig:mean_regression_convergence_mean}. The graphics show that in
this very regular setup the NWE converges faster than our CBE since the NWE 
by construction is a weighted average and as such provides smooth estimates. }
\end{example}
In our next example we consider a scenario, in which the NWE oversmoothes the 
mean regression function and performance worse than the CBE. 
\begin{example}\label{ex:gamma_sin}
\emph{
Again consider $X \sim \Unif(0, 10)$, but this time set
\[ s(x) :=  \max\curbr{1, \sqrt{x}} \rb*{1 + \frac{\sin(10x)}{2}}, \quad 
\theta(x): = \min \curbr{ \max \curbr{1, x}, 6}. \]
\begin{figure}
\centering
\includegraphics[width=\textwidth]{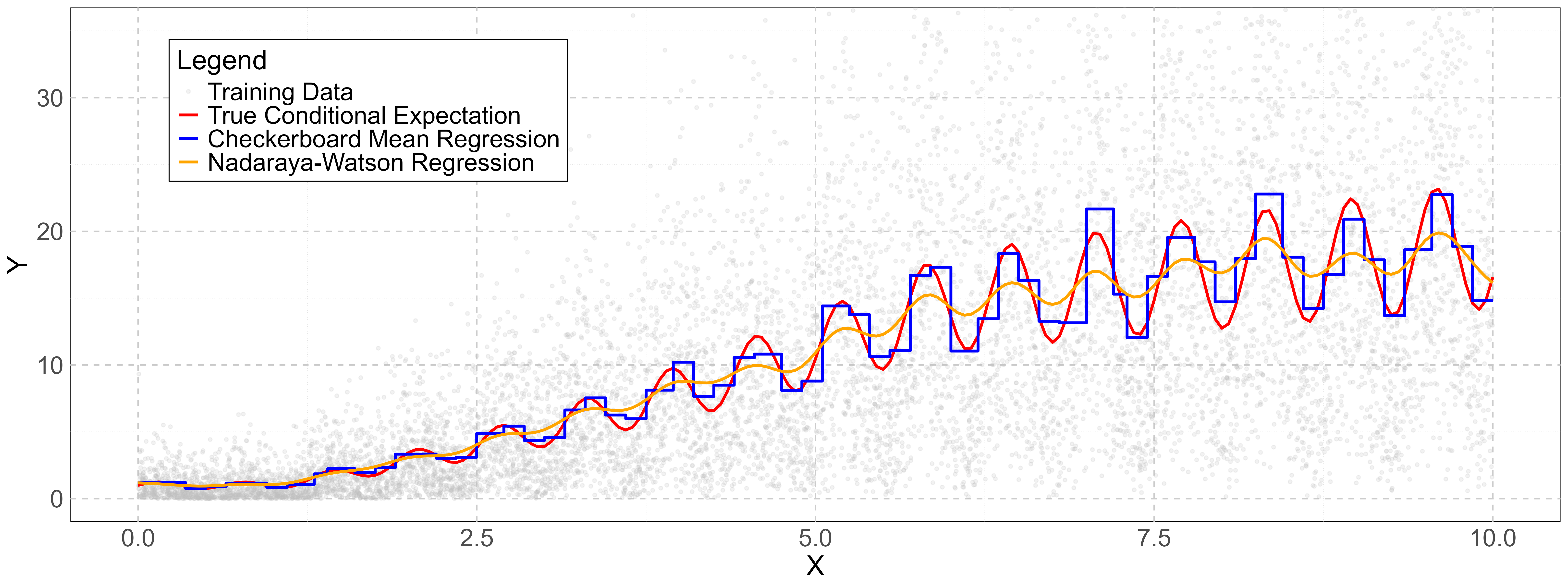}
\caption{Sample of size $n = 10.000$ from the distribution considered in
Example \ref{ex:gamma_sin}. CBE for the regression function (blue), NWE (orange) and true regression 
function (red). }
\label{fig:mean_regression_sin_10000}
\end{figure}
\noindent Sampling from this distribution, calculating, and plotting the CBE, the NWE and 
the true regression function yields the results depicted in Figure~\ref{fig:mean_regression_sin_10000}. Obviously the NWE perfoms poorely 
near the local extrema, whereas the CBE performs better and is only slightly inaccurate  
close to local extrema (which is a natural consequence of the CBE producing 
a step function as estimator for the true regression function). \\
\noindent Repeating the simulations mentioned in Example \ref{ex:gamma_standard} 
for this distribution yields the results gathered in 
Figure~\ref{fig:mean_regression_sin_convergence}. The boxplots insinuate 
that both the CBE and the NWE estimator converge uniformly (and in the $L^1$-sense). 
The upper panel of Figure~\ref{fig:mean_regression_sin_convergence} shows 
that the maximal absolute errors of the CBE and NWE are quite similar. 
The lower panel, however, shows that in terms of the mean absolute error the CBE 
performs better than NWE, confirming the natural idea that in less regular/smooth
setting, the CBE may produce better predictions than the NWE.
\begin{figure}
\centering
\begin{subfigure}[b]{0.9\textwidth}
    \includegraphics[width=\textwidth]{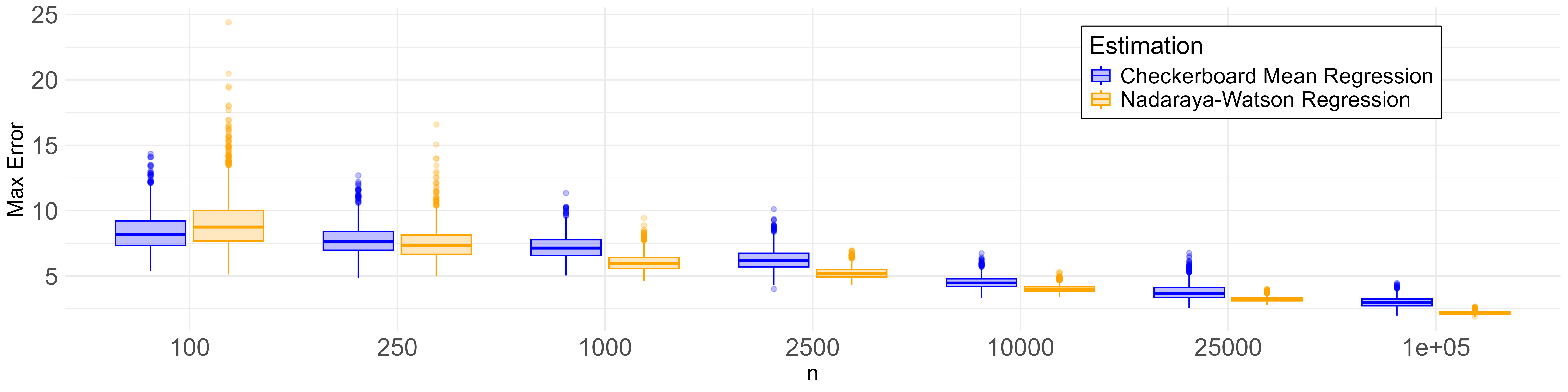}
    \caption{Max Error}
    \label{fig:mean_regression_convergence_sin_max}
\end{subfigure}
~
\begin{subfigure}[b]{0.9\textwidth}
    \includegraphics[width=\textwidth]{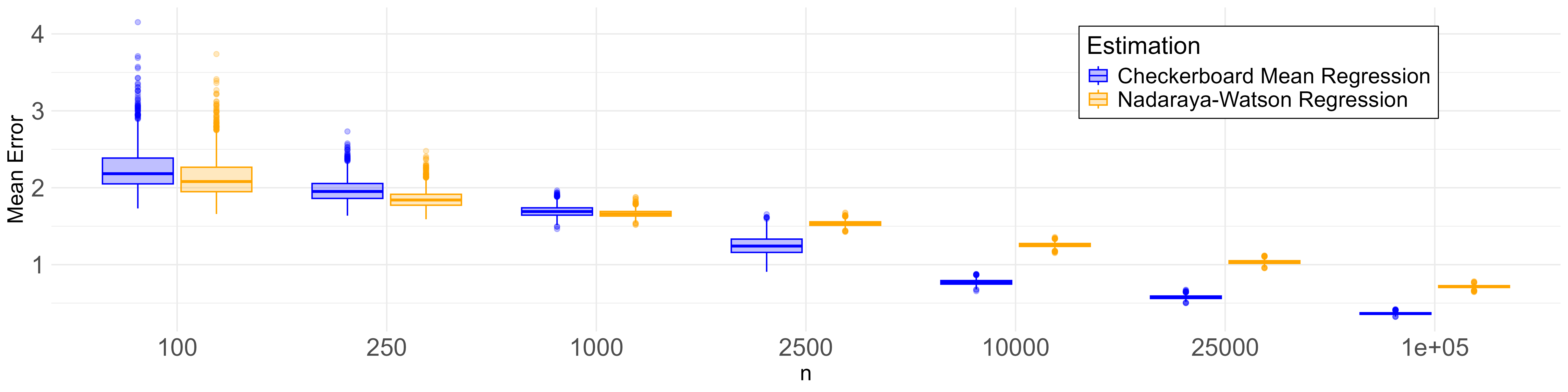}
    \caption{Mean Error}
    \label{fig:mean_regression_convergence_sin_mean}
\end{subfigure}
\caption{(Approximations of the) Mean/max absolute errors of the CBE (blue) and the 
NWE (orange) for the distribution considered in Example \ref{ex:gamma_sin}; 
the mean/max absolute error was calculated by evaluating at $m = 2.000$ randomly 
generated points. The second panel clearly points at
faster convergence of the CBE for the mean absolute error.}
\label{fig:mean_regression_sin_convergence}
\end{figure}}
\end{example}

In the final example concerning mean regression we no longer assume that $X$ is 
uniformly distributed in order to compare the behaviour of the two estimators in areas with 
scarce data. 

\begin{example}\label{ex:gamma_tails}
\emph{
Consider $X \sim \operatorname{Beta}(2, 4)$ and let $Y \;|\; X$ have the same 
distribution as in Example \ref{ex:gamma_standard}. 
Then we have scarce data in the high covariate region. 
Figure~\ref{fig:mean_regression_tails_10000} shows that the CBE produces 
better predictions than the NWE for large values of $x$, implying that 
the CBE may outperform the NWE in regions with spare data. 
Notice that, considering that our method builds upon aggregating the (rank-based) 
empirical copula,
not surprisingly it exhibits more robustness than the NWE. 
\begin{figure}
\centering
\includegraphics[width=\textwidth]{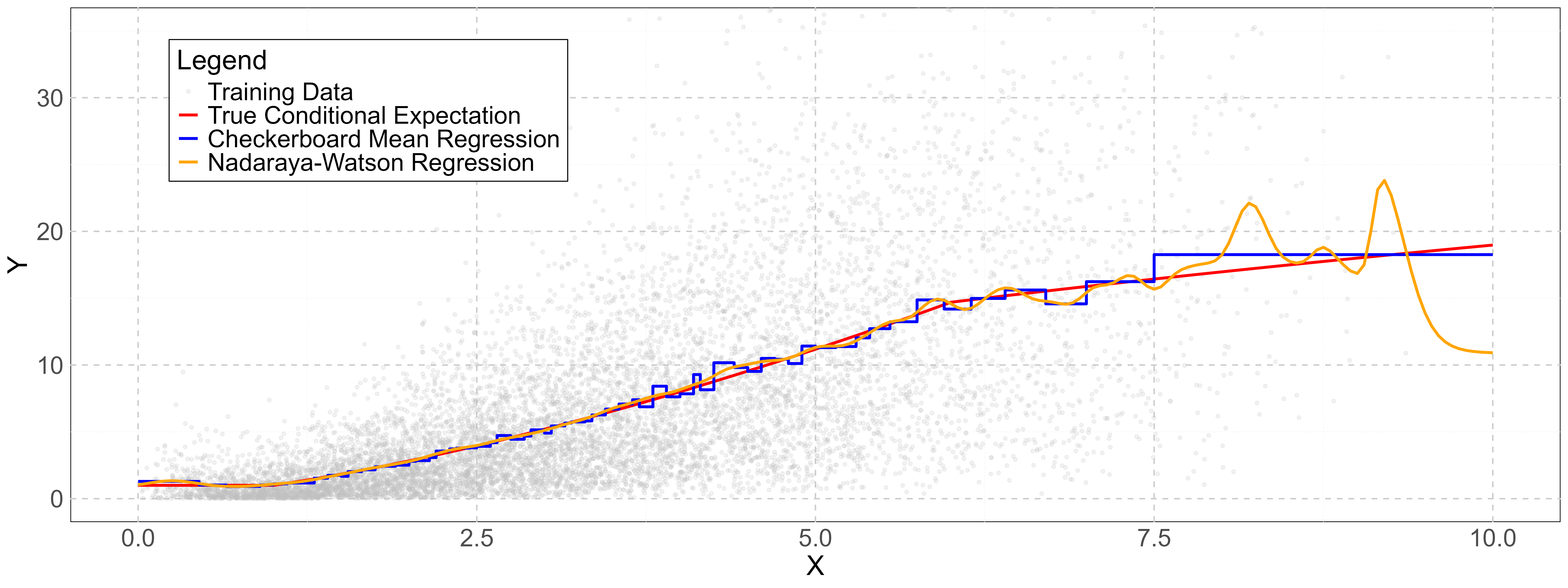}
\caption{Sample of size $n = 10.000$ from the distribution considered in
Example \ref{ex:gamma_tails}. CBE for the regression function (blue), NWE (orange) and true regression 
function (red). }
\label{fig:mean_regression_tails_10000}
\end{figure}
\begin{figure}
\centering
\begin{subfigure}[b]{0.9\textwidth}
    \includegraphics[width=\textwidth]{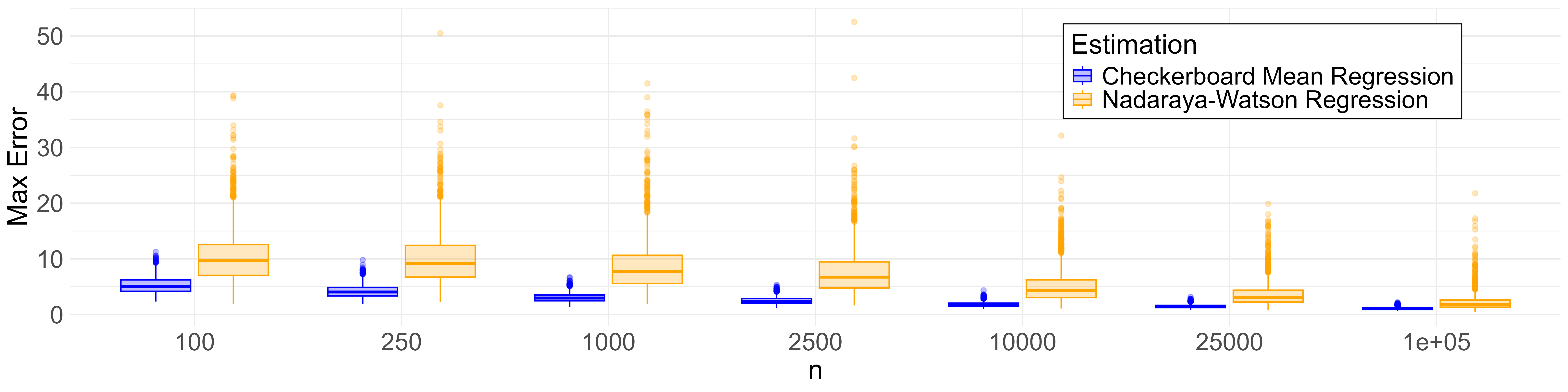}
    \caption{Max Error}
    \label{fig:mean_regression_convergence_tails_max}
\end{subfigure}
~
\begin{subfigure}[b]{0.9\textwidth}
    \includegraphics[width=\textwidth]{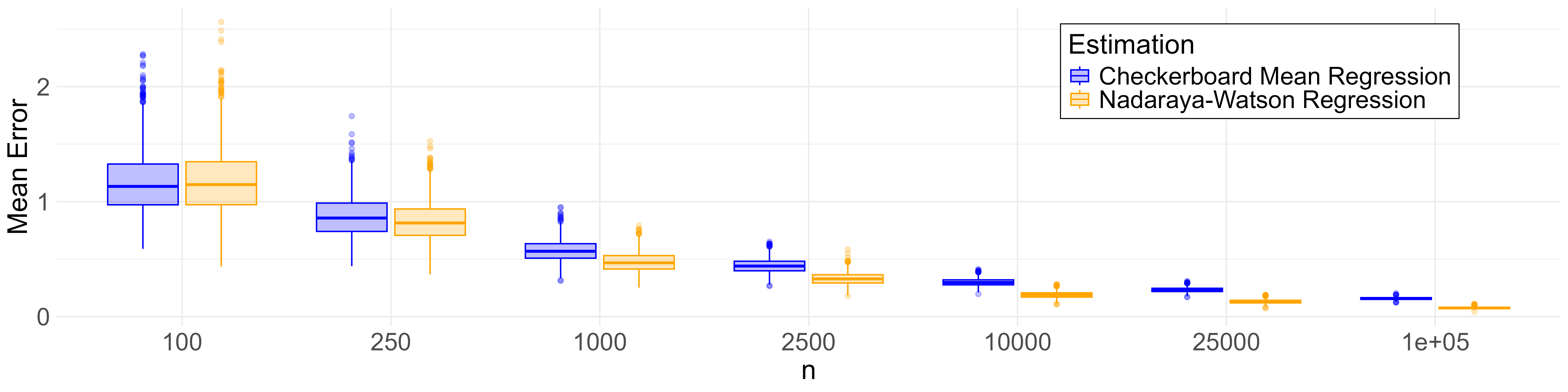}
    \caption{Mean Error}
    \label{fig:mean_regression_convergence_tails_mean}
\end{subfigure}
\caption{(Approximations of the) Mean absolute errors of the CBE (blue) and the 
NWE (orange) for the distribution considered in Example \ref{ex:gamma_tails}; 
the mean/max absolute error was calculated by evaluating at $m = 2.000$ randomly 
generated points. The first panel clearly points at
faster convergence of the CBE for the maximum absolute error.}
\label{fig:mean_regression_tails_convergence}
\end{figure}
Figure~\ref{fig:mean_regression_tails_convergence} illustrates that both the CBE and the 
NWE seem to converge in the maximum and mean sense. 
Moreover, in terms of the maximal absolute error the CBE performs much better 
(see upper panel), 
whereas in mean absolute error the NWE exhibits (slightly) smaller errors.}  
\end{example}

We proceed with a simulation illustrating the behaviour of the CBE
in the context of quantile regression. In this context we compare our 
empi\-rical checkerboard quantile estimator (CBQE) studied in Section~\ref{subsec:quantile_regression} with the Nadaraya-Watson estimator \citep{hall99} for quantile regression (NWQE, for short). 
Notice that for predicting $m$ points our proposed CBQE exhibits a time complexity of $\bigO{m + N(n)n}$, while the NWQE has a time complexity of $\bigO{mn \log(n)}$.
The additional $\log(n)$ term in the NWQE is due to the additional sorting step required when computing. In the CBQE, the sorting is only required in the preparatory step, where $\log(n)$ is dominated by $N(n)$.
\begin{example}\label{ex:gamma_standard_quantile}
\emph{
We again consider the distribution discussed in Example~\ref{ex:gamma_standard} and focus
on the median (i.e., $\tau=\frac{1}{2}$).
Calculating and plotting the CBQE, the NWQE, and the true quantile function yields the results depicted in Figure~\ref{fig:quantile_regression_10000}. 
Similar to Example~\ref{ex:gamma_standard}, we can clearly observe the piecewise constant form of the CBQE contrasting the smooth function produced by the NWQE. 
We can once again observe a difference in accuracy comparing the high variance 
and the low variance region.
\begin{figure}
\centering
\includegraphics[width=\textwidth]{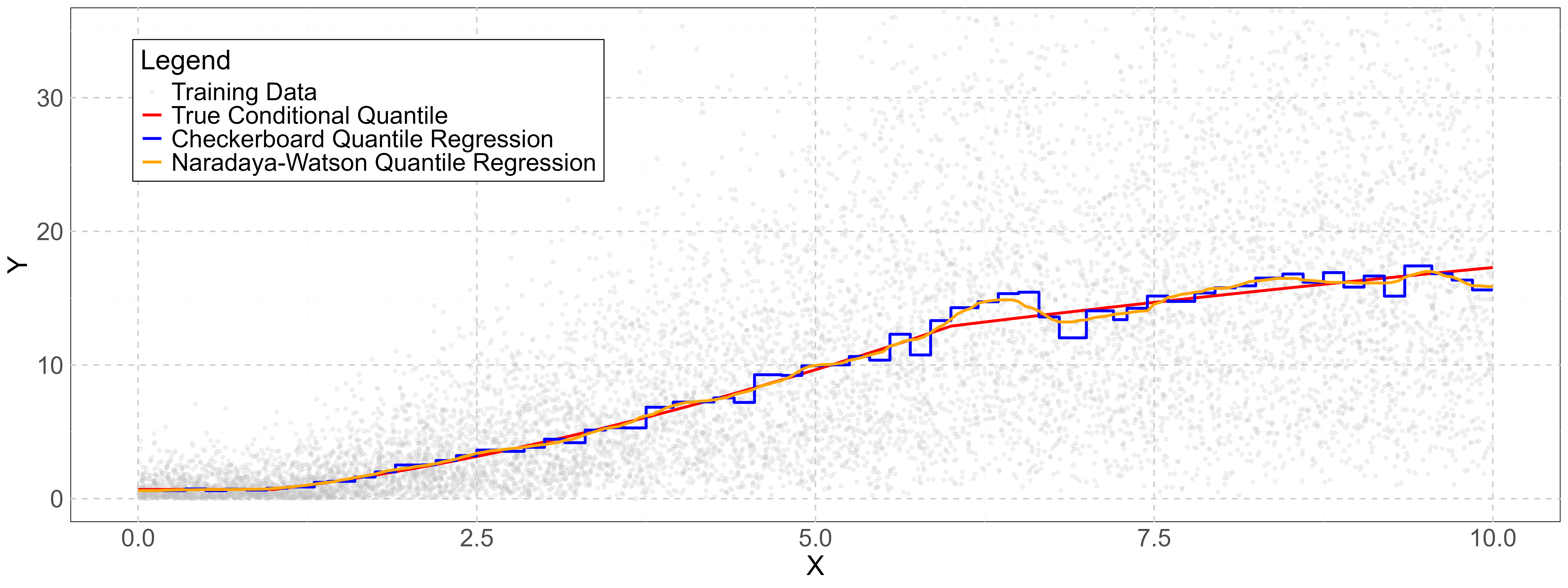}
\caption{Sample of size $n = 10.000$ from the distribution considered in
Example \ref{ex:gamma_standard_quantile}. CBQE for the regression function (blue), NWQE (orange) and true quantile 
function (red). }
\label{fig:quantile_regression_10000}
\end{figure}
\begin{figure}
\centering
\begin{subfigure}[b]{0.9\textwidth}
    \includegraphics[width=\textwidth]{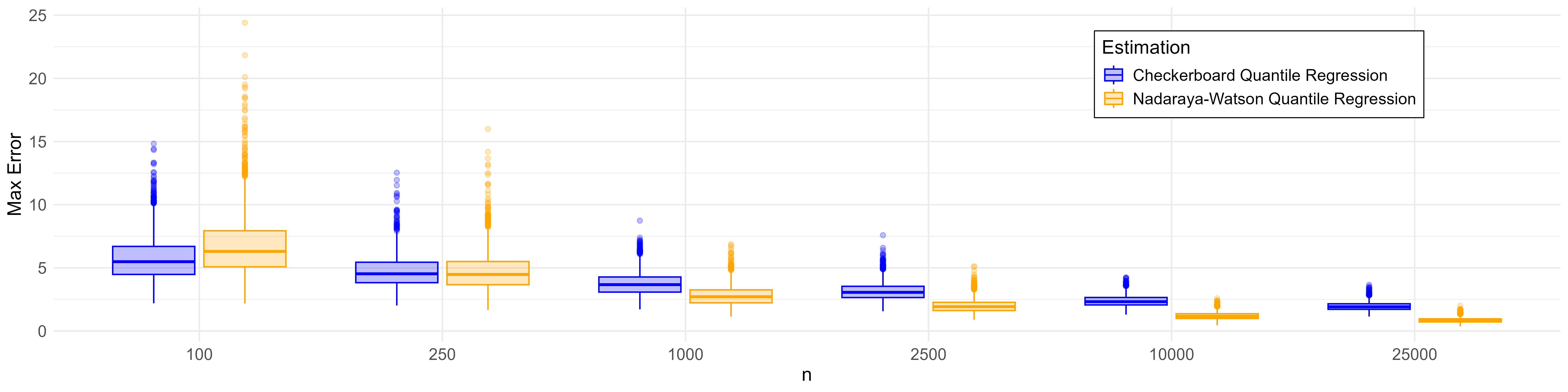}
    \caption{Max Error}
    \label{fig:quantile_regression_convergence_max}
\end{subfigure}
~
\begin{subfigure}[b]{0.9\textwidth}
    \includegraphics[width=\textwidth]{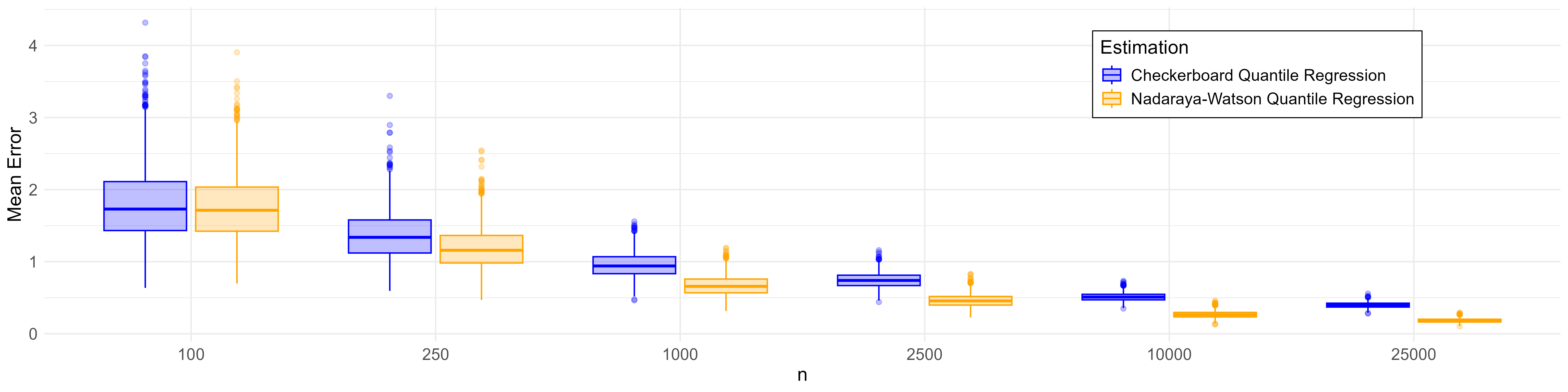}
    \caption{Mean Error}
    \label{fig:quantile_regression_convergence_mean}
\end{subfigure}
\caption{(Approximations of the) Mean absolute errors of the CBQE (blue) and the 
NWQE (orange) for the distribution considered in Example \ref{ex:gamma_standard_quantile}; 
the mean/max absolute error was calculated by evaluating at $m = 2.000$ randomly 
generated points.}
\label{fig:quantile_regression_convergence}
\end{figure}
Repeating the simulations from Example~\ref{ex:gamma_standard}, we obtain Figure~\ref{fig:quantile_regression_convergence}, which can be summarized 
similarly as Figure~\ref{fig:mean_regression_convergence}. }
\end{example}



\begin{remark}
\emph{It is crucial to note that uniform conditional convergence of the empirical 
checkerboard approximation as stated in Theorem~\ref{thm:uniform_convergence} has 
consequences going beyond regression.  
One simple example is the nonparametric estimation for the conditional variance. 
Considering the set-up from Example~\ref{ex:gamma_standard} 
(with adjusted sample size $n = 100.000$ and $s = 0.35$), and calculating the variance 
of the random variable $Z^x$ with distribution function 
$K_{\CB_{N(n)}(E_n)}(F_n(x), [0, G_n(\cdot)])$ on a grid of $x$-values 
yields the step function depicted in blue in 
Figure~\ref{fig:variance_regression_100000} as estimator for the conditional variance;  
the true conditional variance function is depicted in red.
\begin{figure}
\centering
\includegraphics[width=\textwidth]{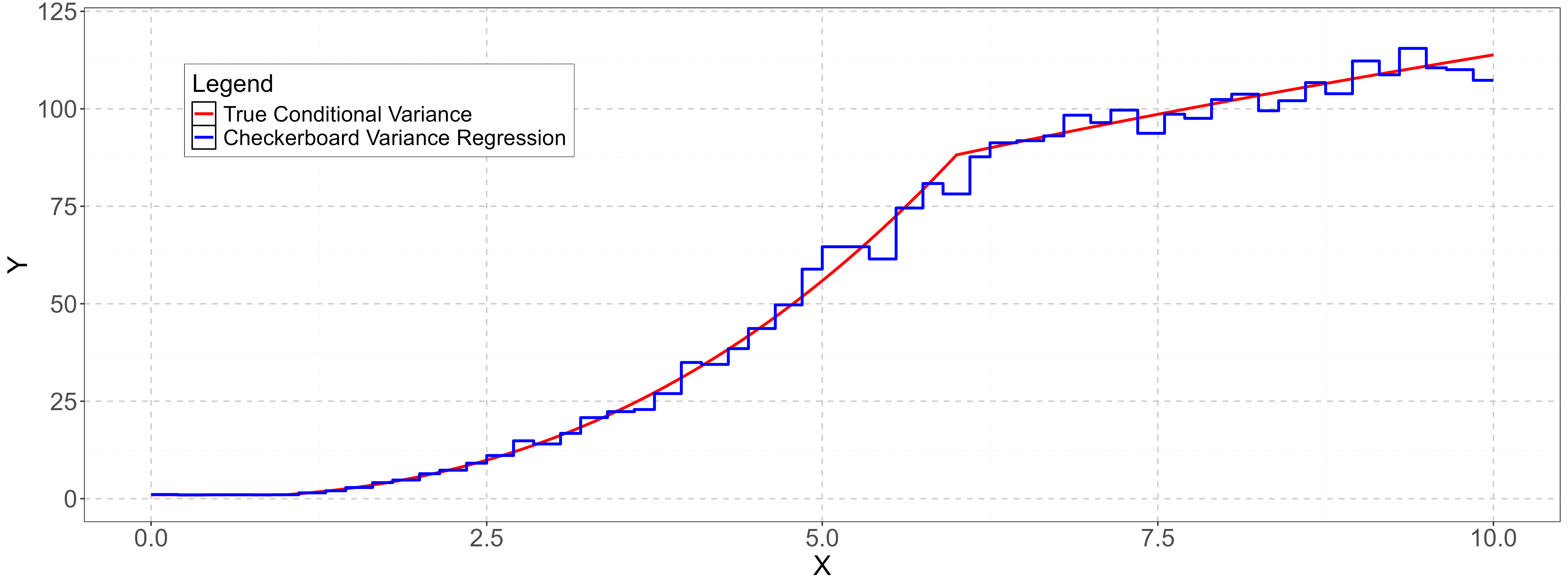}
\caption{Sample of size $n = 100.000$ from the distribution considered in
Example \ref{ex:gamma_standard}. CBE for the conditional variance function (blue) and 
true conditional variance function (red). }
\label{fig:variance_regression_100000}
\end{figure}
We can see that the estimator correctly captures the general trend of the true conditional variance function. Although the speed of convergence seems quite low, this 
underlines the flexibility and the wide range of applications of Theorem~\ref{thm:uniform_convergence}.
} 
\end{remark}

\begin{remark}
\emph{
Notice that moving away from step functions and aiming at smoother estimators 
one could consider combining the empirical checkerboard based estimators and 
smooth the obtained step functions by a smoothing method of choice. 
We have refrained from doing so in the current paper since our main focus is 
on empirical checkerboard approximations and their main properties.}  
\end{remark}


\section{Real Data Example: Insurance Loss vs. Claim-Handling Costs}\label{sec:real_world}

Applying our estimators to a real data example we consider the data 
provided by Insurance Service Offices, analysed in \citet{frees98}, and 
consisting of $n = 1.466$ uncensored observations. 
This dataset is publicly available through the R package \texttt{copula} \citep{hofert25} and is called \texttt{loss}.
We focus on the relationship of the covariate `log-idemnity payment' (loss) and the 
response `log-allocated loss adjustment payment' (ALAE).
Figure~\ref{fig:case_study_plot} depicts the $n = 1.466$ datapoints 
as well as the CBE (blue) and the NWE (orange). We can see that the CBE is much more stable 
in the tails, where the NWE is known to be sensitive to single points/outliers.
\begin{figure}
    \centering
    \includegraphics[width=\linewidth]{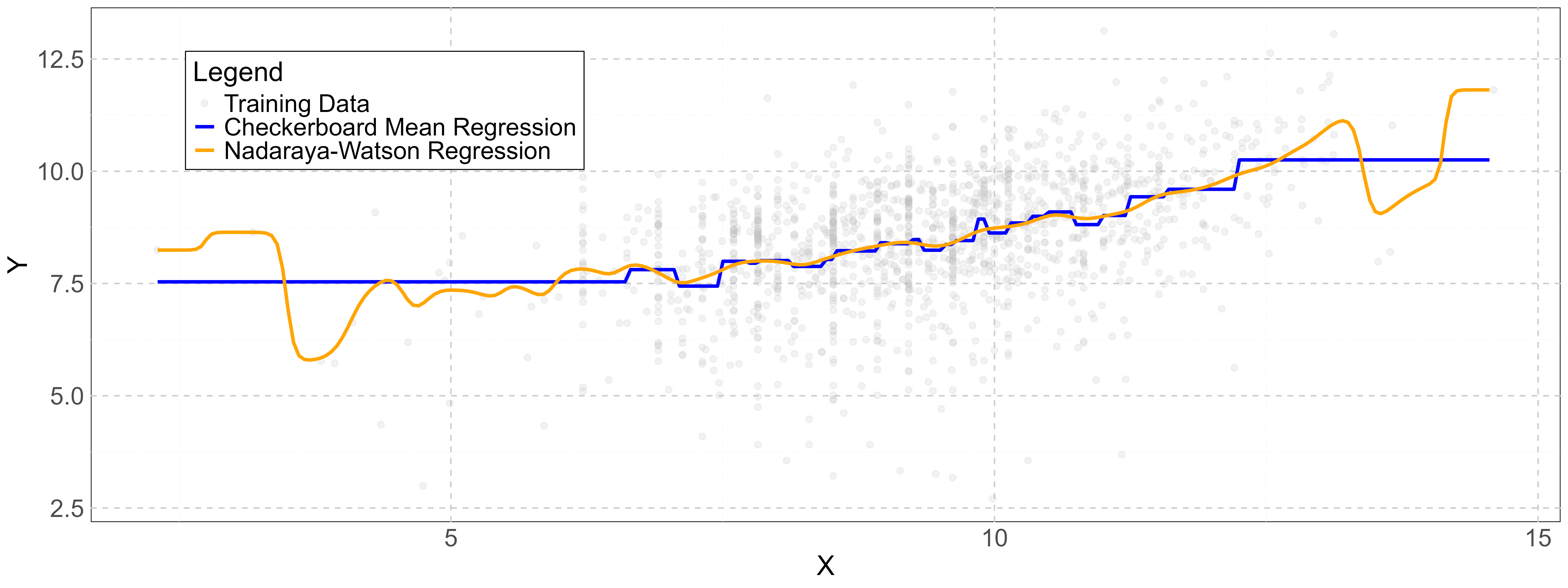}
    \caption{CBE (blue) and NWE (orange) for the insurance data.}
    \label{fig:case_study_plot}
\end{figure}
For comparing the performance of the CBE and the NWE we ran the following 
small experiment $R=10.000$ times: First, we randomly split the data into a training set 
(80\%) and a test set (20\%), compute the CBE and NWE for the training set, and 
compare them with the true observations. 
Figure~\ref{fig:case_study_error_max} depicts the obtained results showing that 
the CBE and the NWE perform similarly. 
\begin{figure}
\centering
\begin{subfigure}{.5\textwidth}
  \centering
  \includegraphics[width=\linewidth]{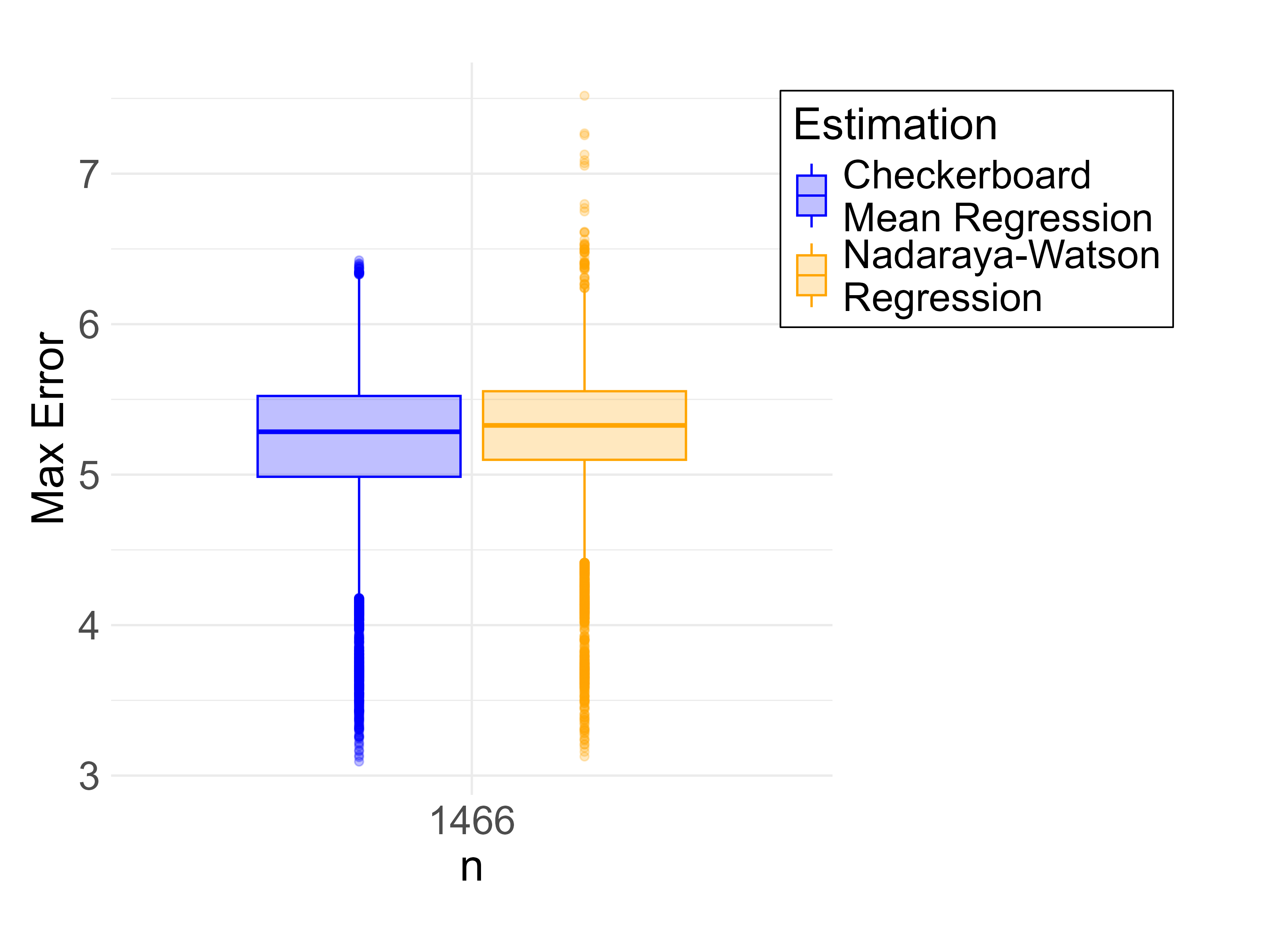}
  \caption{Max Error}
  \label{fig:case_study_error_max}
\end{subfigure}%
\begin{subfigure}{.5\textwidth}
  \centering
  \includegraphics[width=\linewidth]{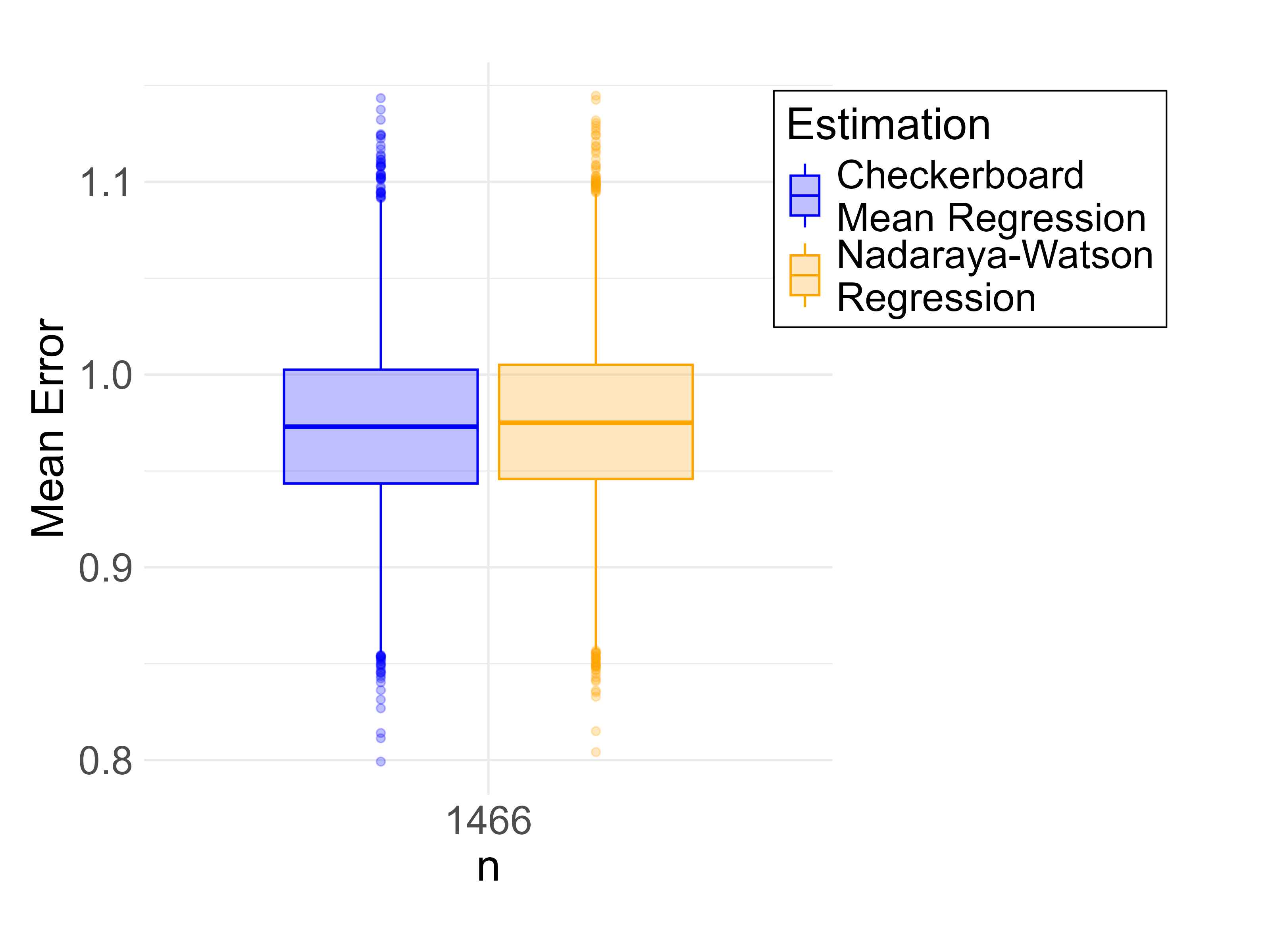}
  \caption{Mean Error}
  \label{fig:case_study_error_mean}
\end{subfigure}
\caption{Maximal/mean absolute errors of the CBE and the NWE obtained by randomly splitting the dataset, training the models on the training set and calculating their errors for the test set
  ($R = 10.000$ runs).}
\label{fig:case_study_error}
\end{figure}

\noindent We observe that our proposed CBE performs better when considering maximal absolute errors, while the precision is similar for the mean absolute error. 
This suggests that we do not lose accuracy compared to the popular NWE, while being  computationally more efficient and more stable in the tails.\\

\noindent \textbf{Acknowledgements.}
Both authors gratefully acknowledge the support of the WISS 2025 project ‘IDA-lab Salzburg’
(20204-WISS/225/348/3-2023 and 20102/F2300464-KZP).

\appendix
\renewcommand{\thesection}{\Alph{section}}


\end{document}